\documentclass[12pt]{amsart}

\usepackage{amsmath}
\usepackage{amssymb}
\usepackage{enumerate}

\makeatletter
\@namedef{subjclassname@2010}{%
  \textup{2010} Mathematics Subject Classification}
\makeatother

\newtheorem{thm}{Theorem}[section]
\newtheorem{cor}[thm]{Corollary}
\newtheorem{lem}[thm]{Lemma}

\theoremstyle{definition}
\newtheorem{defin}[thm]{Definition}
\newtheorem{rem}[thm]{Remark}

\numberwithin{equation}{section}

\begin{document}

%%%%% To ease editing, for IMPAN journals add:

\baselineskip=17pt

%%%%%%%%%%%

%% In the running head, replace first names by initials
%% and give an abbreviation of the title.

\title[Morrey spaces for Schr\"odinger operators with nonnegative potentials]{Morrey spaces for Schr\"odinger operators with nonnegative potentials, fractional integral operators and the Adams inequality on the Heisenberg groups}

\author[H. Wang]{Hua Wang}
\address{School of Mathematics and Systems Science, Xinjiang University, Urumqi 830046, P. R. China\\
}
\email{wanghua@pku.edu.cn}
\date{}

\begin{abstract}
Let $\mathcal L=-\Delta_{\mathbb H^n}+V$ be a Schr\"odinger operator on the Heisenberg group $\mathbb H^n$, where $\Delta_{\mathbb H^n}$ is the sublaplacian on $\mathbb H^n$ and the nonnegative potential $V$ belongs to the reverse H\"older class $RH_s$ with $s\in[Q/2,\infty)$. Here $Q=2n+2$ is the homogeneous dimension of $\mathbb H^n$. For given $\alpha\in(0,Q)$, the fractional integral operator associated with the Schr\"odinger operator $\mathcal L$ is defined by $\mathcal I_{\alpha}={\mathcal L}^{-{\alpha}/2}$. In this article, the author introduces the Morrey space $L^{p,\kappa}_{\rho,\infty}(\mathbb H^n)$ and weak Morrey space $WL^{p,\kappa}_{\rho,\infty}(\mathbb H^n)$ associated with $\mathcal L$, where $(p,\kappa)\in[1,\infty)\times[0,1)$ and $\rho(\cdot)$ is an auxiliary function related to the nonnegative potential $V$. The relation between the fractional integral operator and the maximal operator on the Heisenberg group is established. From this, the author further obtains the Adams (Morrey-Sobolev) inequality on these new spaces. It is shown that the fractional integral operator $\mathcal I_{\alpha}={\mathcal L}^{-{\alpha}/2}$ is bounded from $L^{p,\kappa}_{\rho,\infty}(\mathbb H^n)$ to $L^{q,\kappa}_{\rho,\infty}(\mathbb H^n)$ with $0<\alpha<Q$, $1<p<Q/{\alpha}$, $0<\kappa<1-{(\alpha p)}/Q$ and $1/q=1/p-{\alpha}/{Q(1-\kappa)}$, and bounded from $L^{1,\kappa}_{\rho,\infty}(\mathbb H^n)$ to $WL^{q,\kappa}_{\rho,\infty}(\mathbb H^n)$ with $0<\alpha<Q$, $0<\kappa<1-\alpha/Q$ and $1/q=1-{\alpha}/{Q(1-\kappa)}$.  Moreover, in order to deal with the extreme cases $\kappa\geq 1-{(\alpha p)}/Q$, the author also introduces the spaces $\mathrm{BMO}_{\rho,\infty}(\mathbb H^n)$ and $\mathcal{C}^{\beta}_{\rho,\infty}(\mathbb H^n)$, $\beta\in(0,1]$ associated with $\mathcal L$. In addition, it is proved that $\mathcal I_{\alpha}$ is bounded from $L^{p,\kappa}_{\rho,\infty}(\mathbb H^n)$ to $\mathrm{BMO}_{\rho,\infty}(\mathbb H^n)$ under $\kappa=1-{(\alpha p)}/Q$, and bounded from $L^{p,\kappa}_{\rho,\infty}(\mathbb H^n)$ to $\mathcal{C}^{\beta}_{\rho,\infty}(\mathbb H^n)$ under $\kappa>1-{(\alpha p)}/Q$ and $\beta=\alpha-{(1-\kappa)Q}/p$.
\end{abstract}

\maketitle

\section{Introduction}

\subsection{The Heisenberg group $\mathbb H^n$}
This paper deals with Morrey spaces for Schr\"odinger operators with nonnegative potentials and fractional integral operators on the Heisenberg groups. The Heisenberg group is the most well-known example from the realm of nilpotent Lie groups and plays an important role in several branches of mathematics such as representation theory, partial differential equations, several complex analysis and harmonic analysis. It is a remarkable fact that the Heisenberg group arises in two fundamental but different settings in analysis. On the one hand, it can be identified with the group of translations of the Siegel upper half space in $\mathbb C^{n+1}$ and plays an important role in our understanding of some problems in the complex function theory of the unit ball. On the other hand, it can be realized as the group of unitary operators generated by the position and momentum operators in the context of quantum mechanics.

We write $\mathbb N=\{1,2,3,\dots\}$ for the set of natural numbers and $\mathbb N_0=\{0,1,2,\dots\}$ for the set of natural numbers extended by zero. The sets of real and complex numbers are denoted by $\mathbb R$ and $\mathbb C$, respectively. The \emph{Heisenberg group} $\mathbb H^n$ is a nilpotent Lie group whose underlying manifold is $\mathbb C^n\times\mathbb R$. The group structure (the multiplication law) is defined by
\begin{equation*}
(z,t)\cdot(z',t'):=\Big(z+z',t+t'+2\mathrm{Im}(z\cdot\overline{z'})\Big),
\end{equation*}
where $z=(z_1,z_2,\dots,z_n)$, $z'=(z_1',z_2',\dots,z_n')\in\mathbb C^n$, and
\begin{equation*}
z\cdot\overline{z'}:=\sum_{j=1}^nz_j\overline{z_j'}.
\end{equation*}
Under this multiplication $\mathbb H^n$ becomes a nilpotent unimodular Lie group, the Haar measure on $\mathbb H^n$ being the Lebesgue measure $dzdt$ on $\mathbb C^n\times\mathbb R=\mathbb R^{2n}\times\mathbb R$. The measure of any measurable set $E\subset\mathbb H^n$ is denoted by $|E|$. The corresponding Lie algebra $\mathfrak{h}^n$ is generated by the $(2n+1)$ left-invariant vector fields on $\mathbb H^n$
\begin{equation*}
\begin{cases}
X_j:=\displaystyle\frac{\partial}{\partial x_j}+2y_j\frac{\partial}{\partial t},\quad j=1,2,\dots,n;&\\
Y_j:=\displaystyle\frac{\partial}{\partial y_j}-2x_j\frac{\partial}{\partial t},\quad j=1,2,\dots,n;&\\
T:=\displaystyle\frac{\partial}{\partial t}.&
\end{cases}
\end{equation*}
All non-trivial commutation relations are given by
\begin{equation*}
[X_j,Y_j]=-4T,\quad j=1,2,\dots,n.
\end{equation*}
The sublaplacian $\Delta_{\mathbb H^n}$ is explicitly given by
\begin{equation*}
\Delta_{\mathbb H^n}:=\sum_{j=1}^n\big(X_j^2+Y_j^2\big).
\end{equation*}
It can be easily seen that the inverse element of $u=(z,t)\in\mathbb H^n$ is $u^{-1}=(-z,-t)$, and the identity is the origin $(0,0)$. For each positive number $a>0$, we define the \emph{dilation} on $\mathbb H^n$ by
\begin{equation*}
\delta_a(z,t):=(az,a^2t),\quad a>0.
\end{equation*}
For any given $(z,t)\in\mathbb H^n$, the \emph{homogeneous norm} of $(z,t)$ is given by
\begin{equation*}
|(z,t)|:=\big(|z|^4+t^2\big)^{1/4}.
\end{equation*}
Observe that $|(z,t)^{-1}|=|(z,t)|$ and
\begin{equation*}
\big|\delta_a(z,t)\big|=\big(|az|^4+(a^2t)^2\big)^{1/4}=a|(z,t)|,\quad a>0.
\end{equation*}
In addition, this norm $|\cdot|$ satisfies the triangle inequality and leads to a left-invariant distant $d(u,v)=\big|u^{-1}\cdot v\big|$ for any $u=(z,t)$, $v=(z',t')\in\mathbb H^n$. The ball of radius $r$ centered at $u$ is denoted by
\begin{equation*}
B(u,r):=\big\{v\in\mathbb H^n:d(u,v)<r\big\}.
\end{equation*}
 For $(u,r)\in\mathbb H^n\times(0,\infty)$, it can be shown that the measure of $B(u,r)$ is
\begin{equation*}
|B(u,r)|=r^{Q}\cdot|B(0,1)|,
\end{equation*}
where $Q:=2n+2$ is the \emph{homogeneous dimension} of $\mathbb H^n$ and $|B(0,1)|$ is the measure of the unit ball in $\mathbb H^n$. A direct calculation shows that the measure of the unit ball in $\mathbb H^n$ is
\begin{equation*}
|B(0,1)|=\frac{2\pi^{n+\frac{\,1\,}{2}}\Gamma(\frac{\,n\,}{2})}{(n+1)\Gamma(n)\Gamma(\frac{n+1}{2})}.
\end{equation*}
For any ball $B=B(u,r)$ in $\mathbb H^n$ and $\lambda\in(0,\infty)$, we shall use the notation $\lambda B$ to denote $B(u,\lambda r)$ and use $B^{\complement}$ to denote its complement $\mathbb H^n\backslash B$. It is clear that
\begin{equation}\label{homonorm}
|B(u,\lambda r)|=\lambda^{Q}\cdot|B(u,r)|,\quad (u,r)\in\mathbb H^n\times(0,\infty),\;\lambda\in(0,\infty).
\end{equation}
For a radial function $F$, we have the following integration formula:
\begin{equation}\label{radial}
\int_{\mathbb H^n}F(u)\,du=c\int_0^\infty F(\varrho)\varrho^{Q-1}\,d\varrho,\quad Q=2n+2,
\end{equation}
where $c$ is a positive constant. For more information about the harmonic analysis on the Heisenberg groups, we refer the readers to \cite[Chapter XII]{stein2}, \cite{folland}, \cite{thangavelu} and the references therein.

\subsection{The Schr\"odinger operator $\mathcal L$}
Let $V:\mathbb H^n\rightarrow\mathbb R$ be a nonnegative locally integrable function that belongs to the \emph{reverse H\"older class} $RH_s$ for some exponent $1<s<\infty$; i.e., there exists a positive constant $C$ such that the following reverse H\"older inequality
\begin{equation*}
\left(\frac{1}{|B|}\int_B V(w)^s\,dw\right)^{1/s}\leq C\left(\frac{1}{|B|}\int_B V(w)\,dw\right)
\end{equation*}
holds for every ball $B$ in $\mathbb H^n$. For given $V\in RH_s$ with $s\geq Q/2$ and $V\not\equiv0$, we introduce the \emph{critical radius function} $\rho(u)=\rho(u;V)$ which is given by
\begin{equation}\label{rho}
\rho(u):=\sup\bigg\{r\in(0,\infty):\frac{1}{r^{Q-2}}\int_{B(u,r)}V(w)\,dw\leq1\bigg\},\quad u\in\mathbb H^n,
\end{equation}
where $B(u,r)$ denotes the ball in $\mathbb H^n$ centered at $u$ and with radius $r$. It is well known that this auxiliary function satisfies $0<\rho(u)<\infty$ for any $u\in\mathbb H^n$ under the above assumption on $V$ (see \cite{lu,lin}). We need the following known result concerning the critical radius function \eqref{rho}.
\begin{lem}[\cite{lu}]\label{N0}
Let $\rho$ be as in \eqref{rho}. If $V\in RH_s$ with $s\geq Q/2$, then there exist constants $C_0\geq 1$ and $N_0>0$ such that, for all $u$ and $v$ in $\mathbb H^n$,
\begin{equation}\label{com}
\frac{\,1\,}{C_0}\left[1+\frac{|v^{-1}u|}{\rho(u)}\right]^{-N_0}\leq\frac{\rho(v)}{\rho(u)}\leq C_0\left[1+\frac{|v^{-1}u|}{\rho(u)}\right]^{\frac{N_0}{N_0+1}}.
\end{equation}
\end{lem}
Lemma \ref{N0} is due to Lu \cite{lu} (see also \cite[Lemma 4]{lin}). In the setting of $\mathbb R^n$, this result was given by Shen in \cite[Lemma 1.4]{shen}. As a straightforward consequence of \eqref{com}, we can deduce that for each integer $k\in\mathbb N_0$, the following inequality
\begin{equation}\label{com2}
\left[1+\frac{2^kr}{\rho(v)}\right]\geq \frac{1}{C_0}\left[1+\frac{r}{\rho(u)}\right]^{-\frac{N_0}{N_0+1}}\left[1+\frac{2^kr}{\rho(u)}\right]
\end{equation}
holds for any $v\in B(u,2^kr)$ with $u\in\mathbb H^n$ and $r\in(0,\infty)$, $C_0$ is the same as in \eqref{com}.

Let $V\in RH_s$ with $s\geq Q/2$ and $V\not\equiv0$. For such a potential $V$, we consider the time independent \emph{Schr\"odinger operator} on $\mathbb H^n$ (see \cite{lin}),
\begin{equation}\label{sch}
\mathcal L:=-\Delta_{\mathbb H^n}+V,
\end{equation}
and its associated semigroup $\big\{\mathcal T^{\mathcal L}_s\big\}_{s>0}$,
\begin{equation*}
\mathcal T^{\mathcal L}_sf(u):=e^{-s\mathcal L}f(u)=\int_{\mathbb H^n}P_s(u,v)f(v)\,dv,\quad f\in L^2(\mathbb H^n),~s>0,
\end{equation*}
where $P_s(u,v)$ denotes the kernel of the operator $e^{-s\mathcal L},s>0$.
\subsection{Fractional integral operator}
First we recall the fractional power of the Laplacian operator on $\mathbb R^n$. For given $\alpha\in(0,n)$, the classical fractional integral operator $I^{\Delta}_{\alpha}$ (also referred to as the Riesz potential) is defined by
\begin{equation*}
I^{\Delta}_{\alpha}(f):=(-\Delta)^{-\alpha/2}(f),
\end{equation*}
where $\Delta:=\sum_{j=1}^n\frac{\partial^2}{\partial x_j^2}$ is the standard Laplacian operator on $\mathbb R^n$. Let $\mathcal S(\mathbb R^n)$ be the space of all Schwartz functions. If $f\in\mathcal S(\mathbb R^n)$, then by virtue of the Fourier transform, we have
\begin{equation*}
\begin{cases}
\displaystyle\widehat{I^{\Delta}_{\alpha}f}(\xi)=(2\pi|\xi|)^{-\alpha}\widehat{f}(\xi),\quad \xi\in\mathbb R^n;&\\
\displaystyle\widehat{\big(|\cdot|^{\alpha-n}\big)}(\xi)=\gamma(\alpha)(2\pi|\xi|)^{-\alpha},\quad \xi\in\mathbb R^n,&
\end{cases}
\end{equation*}
where
\begin{equation*}
\gamma(\alpha):=\frac{\pi^{\frac{n}{2}}2^\alpha\Gamma(\frac{\alpha}{2})}{\Gamma(\frac{n-\alpha}{2})}
\end{equation*}
with $\Gamma(\cdot)$ being the usual gamma function. By using the above equations, we get the following expression of $I^{\Delta}_{\alpha}$.
\begin{equation}\label{frac}
I^{\Delta}_{\alpha}f(x)=\frac{1}{\gamma(\alpha)}\int_{\mathbb R^n}\frac{f(y)}{|x-y|^{n-\alpha}}\,dy=\frac{1}{\gamma(\alpha)}\bigg(\frac{1}{|\cdot|^{n-\alpha}}\ast f\bigg)(x),\quad x\in\mathbb R^n.
\end{equation}
It is well known that the Hardy-Littlewood-Sobolev theorem states that the Riesz potential operator $I^{\Delta}_{\alpha}$ is bounded from $L^p(\mathbb R^n)$ into $L^q(\mathbb R^n)$ for $0<\alpha<n$, $1<p<n/{\alpha}$ and $1/q=1/p-{\alpha}/n$. Also we know that $I^{\Delta}_{\alpha}$ is bounded from $L^1(\mathbb R^n)$ into $WL^q(\mathbb R^n)$ for $0<\alpha<n$ and $q=n/{(n-\alpha)}$ (see, for example, \cite{stein}).

The classical Morrey space $M^{p,\lambda}$ was originally introduced by Morrey in \cite{morrey} to study the local behavior of solutions to second order elliptic partial differential equations. Since then, this space was systematically developed by many authors. Nowadays this space has been studied intensively and widely used in analysis, geometry, mathematical physics and other related fields. For the properties and applications of classical Morrey space, we refer the readers to \cite{adams1,adams2,adams3,fazio1,fazio2,taylor} and the references therein. We denote by $M^{p,\lambda}(\mathbb R^n)$ the Morrey space, the space of all $p$-locally integrable functions $f$ on $\mathbb R^n$ such that
\begin{equation*}
\begin{split}
\|f\|_{M^{p,\lambda}(\mathbb R^n)}:=&\sup_{x\in\mathbb R^n,r>0}r^{-\lambda/p}\|f\|_{L^p(B(x,r))}\\
=&\sup_{x\in\mathbb R^n,r>0}r^{-\lambda/p}\bigg(\int_{B(x,r)}|f(y)|^p\,dy\bigg)^{1/p}<\infty,
\end{split}
\end{equation*}
where $1\leq p<\infty$ and $0\leq\lambda\leq n$. Note that $M^{p,0}(\mathbb R^n)=L^p(\mathbb R^n)$ and $M^{p,n}(\mathbb R^n)=L^\infty(\mathbb R^n)$ by the Lebesgue differentiation theorem. If $\lambda<0$ or $\lambda>n$, then $M^{p,\lambda}(\mathbb R^n)=\Theta$, where $\Theta$ is the set of all functions equivalent to 0 on $\mathbb R^n$. We also denote by $WM^{p,\lambda}(\mathbb R^n)$ the weak Morrey space, which consists of all measurable functions $f$ on $\mathbb R^n$ such that
\begin{equation*}
\begin{split}
\|f\|_{WM^{p,\lambda}(\mathbb R^n)}:=&\sup_{x\in\mathbb R^n,r>0}r^{-\lambda/p}\|f\|_{WL^p(B(x,r))}\\
=&\sup_{x\in\mathbb R^n,r>0}r^{-\lambda/p}\sup_{\sigma>0}\sigma\big|\big\{y\in B(x,r):|f(y)|>\sigma\big\}\big|^{1/p}<\infty.
\end{split}
\end{equation*}
The boundedness of the Riesz potential $I^{\Delta}_{\alpha}$ in Morrey spaces was originally studied by Adams \cite{adams} in 1975. His results can be summarized as follows:
\begin{thm}[Adams \cite{adams}]
Let $0<\alpha<n$, $1\leq p<n/{\alpha}$, $0<\lambda<n-\alpha p$ and $1/q=1/p-\alpha/{(n-\lambda)}$. Then for $p>1$ the Riesz potential $I^{\Delta}_{\alpha}$ is bounded from $M^{p,\lambda}(\mathbb R^n)$ into $M^{q,\lambda}(\mathbb R^n)$, and for $p=1$ the Riesz potential $I^{\Delta}_{\alpha}$ is bounded from $M^{1,\lambda}(\mathbb R^n)$ into $WM^{q,\lambda}(\mathbb R^n)$.
\end{thm}
\begin{rem}
The above theorem can be considered as the analogue of the Hardy-Littlewood-Sobolev theorem in Morrey spaces, and is often called \emph{the Adams inequality} (Morrey-Sobolev inequality).
\end{rem}
Next we are going to discuss the fractional integral operators on the Heisenberg groups. For given $\alpha\in(0,Q)$ with $Q=2n+2$, the fractional integral operator $I_{\alpha}$ (also referred to as the Riesz potential) is defined by (see \cite{xiao})
\begin{equation}\label{frac2}
I_{\alpha}(f):=(-\Delta_{\mathbb H^n})^{-\alpha/2}(f),
\end{equation}
where $\Delta_{\mathbb H^n}$ is the sublaplacian on $\mathbb H^n$ defined above. Let $f$ and $g$ be integrable functions defined on $\mathbb H^n$. Define the \emph{convolution} $f\ast g$ by
\begin{equation*}
(f\ast g)(u):=\int_{\mathbb H^n}f(v)g(v^{-1}u)\,dv.
\end{equation*}
We denote by $H_s(u)$ the convolution kernel of heat semigroup $\big\{T_s\big\}_{s>0}$. Namely,
\begin{equation*}
T_sf(u):=e^{s\Delta_{\mathbb H^n}}f(u)=\int_{\mathbb H^n}H_s(v^{-1}u)f(v)\,dv.
\end{equation*}
For any $u=(z,t)\in\mathbb H^n$, it was proved in \cite[Theorem 4.2]{xiao} that $I_{\alpha}$ can be expressed by the following formula:
\begin{equation}\label{frac3}
\begin{split}
I_{\alpha}f(u)&=\frac{1}{\Gamma(\alpha/2)}\int_0^{\infty}e^{s\Delta_{\mathbb H^n}}f(u)\,s^{\alpha/2-1}ds\\
&=\frac{1}{\Gamma(\alpha/2)}\int_0^{\infty}\big(H_s*f\big)(u)\,s^{\alpha/2-1}ds.
\end{split}
\end{equation}
For any $u=(z,t)\in\mathbb H^n$, it is well known that the heat kernel $H_s(u)$ has the explicit expression:
\begin{equation*}
H_s(z,t)=(2\pi)^{-1}(4\pi)^{-n}\int_{\mathbb R}\bigg(\frac{|\lambda|}{\sinh|\lambda|s}\bigg)^n\exp\left\{-\frac{|\lambda||z|^2}{4}\coth|\lambda|s-i\lambda s\right\}d\lambda.
\end{equation*}
We consider the heat equation associated to the sublaplacian
\begin{equation*}
\partial_sF(u,s)=\Delta_{\mathbb H^n}F(u,s),\quad(u,s)\in\mathbb H^n\times(0,\infty)
\end{equation*}
with the initial condition $F(u,0)=f(u)$. In fact, the function $H_s(u)$ stated above exists as a solution to the heat equation. Moreover, it satisfies the following Gaussian upper bound estimate (see \cite{jerison} for instance)
\begin{equation}\label{heatkernel}
0\leq H_s(u)\leq C\cdot s^{-Q/2}\exp\bigg(-\frac{|u|^2}{As}\bigg),
\end{equation}
where the positive constants $C$ and $A$ are independent of $s\in(0,\infty)$ and $u\in\mathbb H^n$.

The Hardy-Littlewood-Sobolev theorem on the Heisenberg group is established in \cite{xiao}.
\begin{thm}[\cite{xiao}]
Let $0<\alpha<Q$, $1\leq p<Q/{\alpha}$ and $1/q=1/p-\alpha/{Q}$. Then for $p>1$ the Riesz potential $I_{\alpha}$ is bounded from $L^{p}(\mathbb H^n)$ into $L^{q}(\mathbb H^n)$, and for $p=1$ the Riesz potential $I_{\alpha}$ is bounded from $L^{1}(\mathbb H^n)$ into $WL^{q}(\mathbb H^n)$.
\end{thm}

\begin{rem}
In \cite{xiao}, they also showed that $1/q=1/p-\alpha/Q$ is a necessary condition for ensuring $\|I_{\alpha}(f)\|_{L^q(\mathbb H^n)}\leq C\|f\|_{L^p(\mathbb H^n)}$ when $1<p<Q/\alpha$, and $1/q=1-\alpha/Q$ is also a necessary condition for ensuring $\|I_{\alpha}(f)\|_{WL^q(\mathbb H^n)}\leq C\|f\|_{L^1(\mathbb H^n)}$.
\end{rem}
For $1\leq p<\infty$, let $L^p(\mathbb H^n)$ be the Lebesgue space of all measurable functions $f$ on $\mathbb H^n$ such that
\begin{equation*}
\big\|f\big\|_{L^p(\mathbb H^n)}:=\bigg(\int_{\mathbb H^n}|f(u)|^p\,du\bigg)^{1/p}<\infty.
\end{equation*}
The weak Lebesgue space $WL^p(\mathbb H^n)$ consists of all measurable functions $f$ on $\mathbb H^n$ such that
\begin{equation*}
\big\|f\big\|_{WL^p(\mathbb H^n)}:=
\sup_{\lambda>0}\lambda\cdot\big|\big\{u\in\mathbb H^n:|f(u)|>\lambda\big\}\big|^{1/p}<\infty.
\end{equation*}

Let $\mathcal L$ be the Schr\"odinger operator on $\mathbb H^n$ as in \eqref{sch}. Since $V\geq0$ and $V\in L^1_{\mathrm{loc}}(\mathbb H^n)$, $\mathcal L$ generates a $(C_0)$ contraction semigroup $\big\{\mathcal T^{\mathcal L}_s\big\}_{s>0}=\big\{e^{-s\mathcal L}\big\}_{s>0}$. Let $P_s(u,v)$ denote the kernel of the semigroup $\big\{e^{-s\mathcal L}\big\}_{s>0}$. By the \emph{Trotter product formula} and \eqref{heatkernel}, we know that the kernel $P_s(u,v)$ satisfies the Gaussian upper bound
\begin{equation}\label{heat}
0\leq P_s(u,v)\leq H_s(v^{-1}u)\leq C\cdot s^{-Q/2}\exp\bigg(-\frac{|v^{-1}u|^2}{As}\bigg),\quad s>0.
\end{equation}
Moreover, this estimate \eqref{heat} can be improved when $V$ belongs to the reverse H\"older class $RH_s$ for some $s\geq Q/2$. The auxiliary function $\rho(u)$ arises naturally in this context.
\begin{lem}\label{ker1}
Let $V\in RH_s$ with $s\geq Q/2$, and let $\rho(u)$ be the auxiliary function determined by $V$. For every positive integer $N\in\mathbb N$, there exists a positive constant $C_N>0$ such that, for all $u$ and $v$ in $\mathbb H^n$,
\begin{equation*}
0\leq P_s(u,v)\leq C_N\cdot s^{-Q/2}\exp\bigg(-\frac{|v^{-1}u|^2}{As}\bigg)\bigg[1+\frac{\sqrt{s\,}}{\rho(u)}+\frac{\sqrt{s\,}}{\rho(v)}\bigg]^{-N},\quad s>0.
\end{equation*}
\end{lem}
This estimate of $P_s(u,v)$ is better than \eqref{heat}, which was given by Lin and Liu in \cite[Lemma 7]{lin}. In the setting of $\mathbb R^n$, this result can be found in \cite[Proposition 2]{dziu}.

Inspired by \eqref{frac2} and \eqref{frac3}, for given $\alpha\in(0,Q)$, the \emph{$\mathcal L$-fractional integral operator} or \emph{$\mathcal L$-Riesz potential} on the Heisenberg group is naturally defined by (see \cite{jiang} and \cite{jiang2})
\begin{equation}\label{riesz}
\begin{split}
\mathcal I_{\alpha}(f)(u)&:={\mathcal L}^{-{\alpha}/2}f(u)\\
&=\frac{1}{\Gamma(\alpha/2)}\int_0^{\infty}e^{-s\mathcal L}f(u)\,s^{\alpha/2-1}ds.
\end{split}
\end{equation}
Recall that in the setting of $\mathbb R^n$, this integral operator was first introduced by Dziuba\'{n}ski et al.\cite{dziu}, and studied extensively by many authors in \cite{bong,bong1,bong2,bui,pan,tang}. In this article we shall be interested in the behavior of the fractional integral operator $\mathcal I_{\alpha}$ associated with Schr\"odinger operator $\mathcal L$ on $\mathbb H^n$. First we are going to establish strong-type and weak-type estimates of the $\mathcal L$-fractional integral operator $\mathcal I_{\alpha}$ on the Lebesgue spaces. We now claim that the following inequality
\begin{equation}\label{claim}
|\mathcal I_{\alpha}f(u)|\leq C\int_{\mathbb H^n}|f(v)|\frac{1}{|v^{-1}u|^{Q-\alpha}}\,dv=C\big(|f|*|\cdot|^{\alpha-Q}\big)(u)
\end{equation}
holds for all $u\in\mathbb H^n$. Let us verify \eqref{claim}. To do so, let $\mathcal K_{\alpha}(u,v)$ denote the kernel of the $\mathcal L$-fractional integral operator $\mathcal I_{\alpha}$. Then we have
\begin{equation*}
\begin{split}
\int_{\mathbb H^n}\mathcal K_{\alpha}(u,v)f(v)\,dv&=\mathcal I_{\alpha}f(u)={\mathcal L}^{-{\alpha}/2}f(u)\\
&=\frac{1}{\Gamma(\alpha/2)}\int_0^{\infty}e^{-s\mathcal L}f(u)\,s^{\alpha/2-1}ds\\
&=\int_0^{\infty}\bigg[\frac{1}{\Gamma(\alpha/2)}\int_{\mathbb H^n}P_s(u,v)f(v)\,dv\bigg]s^{\alpha/2-1}ds\\
&=\int_{\mathbb H^n}\bigg[\frac{1}{\Gamma(\alpha/2)}\int_0^{\infty}P_s(u,v)\,s^{\alpha/2-1}ds\bigg]f(v)\,dv.
\end{split}
\end{equation*}
Hence, the kernel $\mathcal K_{\alpha}(u,v)$ can be written as
\begin{equation*}
\mathcal K_{\alpha}(u,v)=\frac{1}{\Gamma(\alpha/2)}\int_0^{\infty}P_s(u,v)\,s^{\alpha/2-1}ds.
\end{equation*}
Moreover, by using \eqref{heat}, we can deduce that
\begin{equation*}
\begin{split}
\big|\mathcal K_{\alpha}(u,v)\big|&\leq\frac{C}{\Gamma(\alpha/2)}\int_0^{\infty}\exp\bigg(-\frac{|v^{-1}u|^2}{As}\bigg)s^{\alpha/2-Q/2-1}ds\\
&=\frac{C}{\Gamma(\alpha/2)}\cdot\frac{1}{|v^{-1}u|^{Q-\alpha}}\int_0^{\infty}e^{-t}\,t^{(Q/2-\alpha/2)-1}dt\\
&=C\cdot\frac{\Gamma(Q/2-\alpha/2)}{\Gamma(\alpha/2)}\cdot\frac{1}{|v^{-1}u|^{Q-\alpha}},
\end{split}
\end{equation*}
where in the second step we have used a change of variables. Thus \eqref{claim} holds. According to Theorems 4.4 and 4.5 in \cite{xiao}, one can get the Hardy-Littlewood-Sobolev theorem for $\mathcal I_{\alpha}$ on $\mathbb H^n$.
\begin{thm}\label{strong}
Let $0<\alpha<Q$ and $1\leq p<Q/{\alpha}$. Define $1<q<\infty$ by the relation $1/q=1/p-{\alpha}/Q$. Then the following statements are valid:
\begin{enumerate}
  \item if $p>1$, then $\mathcal I_{\alpha}$ is bounded from $L^p(\mathbb H^n)$ into $L^q(\mathbb H^n);$
  \item if $p=1$, then $\mathcal I_{\alpha}$ is bounded from $L^1(\mathbb H^n)$ into $WL^q(\mathbb H^n)$.
\end{enumerate}
\end{thm}

This article is organized as follows. In Section \ref{sec2}, we will give the definitions of Morrey space and weak Morrey space associated with Schr\"odinger operator on $\mathbb H^n$ and state our main results:Theorems \ref{mainthm:1}, \ref{mainthm:2}, \ref{mainthm:3} and Corollary \ref{mainthm:4}. Section \ref{sec3} is devoted to establishing some estimates for the kernel of the $\mathcal L$-fractional integral operator and proving Theorems \ref{mainthm:1} and \ref{mainthm:2}. In Section \ref{sec4}, we will study certain extreme cases and give the proof of Theorem \ref{mainthm:3}.

Throughout this article, we denote by $C$ a positive constant which is independent of the main parameters, but it may vary from line to line. We also use $C_{\alpha,\beta,\dots}$ to denote a positive constant depending on the parameters $\alpha,\beta,\dots$. The symbol $f\lesssim g$ means that $f\leq Cg$. If $f\lesssim g$ and $g\lesssim f$, then we write $f\approx g$ to denote the equivalence of $f$ and $g$. For any $p\in(1,\infty)$, the notation $p'$ denotes its conjugate number, namely, $1/p+1/{p'}=1$.

\section{Definitions and Main results}\label{sec2}
In this section, we introduce some kinds of Morrey spaces associated with the Schr\"odinger operator $\mathcal L$ on $\mathbb H^n$, and then give our main results.
\begin{defin}
Let $\rho$ be the auxiliary function determined by $V\in RH_s$ with $s\geq Q/2$. Let $1\leq p<\infty$ and $0\leq\kappa<1$. For any given $0<\theta<\infty$, the Morrey space $L^{p,\kappa}_{\rho,\theta}(\mathbb H^n)$ is defined to be the set of all $p$-locally integrable functions $f$ on $\mathbb H^n$ such that
\begin{equation}\label{morrey1}
\bigg(\frac{1}{|B|^{\kappa}}\int_B|f(u)|^p\,du\bigg)^{1/p}
\leq C\cdot\left[1+\frac{r}{\rho(u_0)}\right]^{\theta}
\end{equation}
holds for every ball $B=B(u_0,r)$ in $\mathbb H^n$, $u_0$ and $r$ denote the center and radius of $B$, respectively. A norm for $f\in L^{p,\kappa}_{\rho,\theta}(\mathbb H^n)$, denoted by $\|f\|_{L^{p,\kappa}_{\rho,\theta}(\mathbb H^n)}$, is given by the infimum of the constants in \eqref{morrey1}, or equivalently,
\begin{equation*}
\big\|f\big\|_{L^{p,\kappa}_{\rho,\theta}(\mathbb H^n)}:=\sup_{B(u_0,r)}\left[1+\frac{r}{\rho(u_0)}\right]^{-\theta}
\bigg(\frac{1}{|B|^{\kappa}}\int_B\big|f(u)\big|^p\,du\bigg)^{1/p}
<\infty,
\end{equation*}
where the supremum is taken over all balls $B=B(u_0,r)$ in $\mathbb H^n$. Define
\begin{equation*}
L^{p,\kappa}_{\rho,\infty}(\mathbb H^n):=\bigcup_{\theta>0}L^{p,\kappa}_{\rho,\theta}(\mathbb H^n).
\end{equation*}
\end{defin}

\begin{defin}
Let $\rho$ be the auxiliary function determined by $V\in RH_s$ with $s\geq Q/2$. Let $1\leq p<\infty$ and $0\leq\kappa<1$. For any given $0<\theta<\infty$, the weak Morrey space $WL^{p,\kappa}_{\rho,\theta}(\mathbb H^n)$ is defined to be the set of all measurable functions $f$ on $\mathbb H^n$ such that
\begin{equation*}
\frac{1}{|B|^{\kappa/p}}\sup_{\lambda>0}\lambda\cdot\big|\big\{u\in B:|f(u)|>\lambda\big\}\big|^{1/p}
\leq C\cdot\left[1+\frac{r}{\rho(u_0)}\right]^{\theta}
\end{equation*}
holds for every ball $B=B(u_0,r)$ in $\mathbb H^n$, or equivalently,
\begin{equation*}
\big\|f\big\|_{WL^{p,\kappa}_{\rho,\theta}(\mathbb H^n)}:=\sup_{B(u_0,r)}\left[1+\frac{r}{\rho(u_0)}\right]^{-\theta}\frac{1}{|B|^{\kappa/p}}
\sup_{\lambda>0}\lambda\cdot\big|\big\{u\in B:|f(u)|>\lambda\big\}\big|^{1/p}<\infty.
\end{equation*}
Correspondingly, we define
\begin{equation*}
WL^{p,\kappa}_{\rho,\infty}(\mathbb H^n):=\bigcup_{\theta>0}WL^{p,\kappa}_{\rho,\theta}(\mathbb H^n).
\end{equation*}
\end{defin}
\begin{rem}
(i) Obviously, if we take $\theta=0$ or $V\equiv0$, then this Morrey space $L^{p,\kappa}_{\rho,\theta}(\mathbb H^n)$ (or weak Morrey space $WL^{p,\kappa}_{\rho,\theta}(\mathbb H^n)$) is just the Morrey space $L^{p,\kappa}(\mathbb H^n)$ (or weak Morrey space $WL^{p,\kappa}(\mathbb H^n)$), which was defined and studied by Guliyev et al. \cite{guliyev}.

(ii) According to the above definitions, one has
\begin{align}
L^{p,\kappa}(\mathbb H^n)\subset L^{p,\kappa}_{\rho,\theta_1}(\mathbb H^n)\subset L^{p,\kappa}_{\rho,\theta_2}(\mathbb H^n);& \label{22}\\
WL^{p,\kappa}(\mathbb H^n)\subset WL^{p,\kappa}_{\rho,\theta_1}(\mathbb H^n)\subset WL^{p,\kappa}_{\rho,\theta_2}(\mathbb H^n),& \label{23}
\end{align}
for $0<\theta_1<\theta_2<\infty$. Hence,
\begin{equation*}
L^{p,\kappa}(\mathbb H^n)\subset L^{p,\kappa}_{\rho,\infty}(\mathbb H^n)\quad \mathrm{and} \quad WL^{p,\kappa}(\mathbb H^n)\subset WL^{p,\kappa}_{\rho,\infty}(\mathbb H^n)
\end{equation*}
for all $(p,\kappa)\in[1,\infty)\times[0,1)$.

(iii) We define a norm on the space $L^{p,\kappa}_{\rho,\infty}(\mathbb H^n)$, which makes it into a Banach space. In view of \eqref{22}, for any given $f\in L^{p,\kappa}_{\rho,\infty}(\mathbb H^n)$, let
\begin{equation*}
\theta^*:=\inf\big\{\theta>0:f\in L^{p,\kappa}_{\rho,\theta}(\mathbb H^n)\big\}.
\end{equation*}
Now define the functional $\|\cdot\|_{\star}$ by
\begin{equation}\label{newnorm}
\|f\|_{\star}=\big\|f\big\|_{L^{p,\kappa}_{\rho,\infty}(\mathbb H^n)}:=\big\|f\big\|_{L^{p,\kappa}_{\rho,\theta^*}(\mathbb H^n)}.
\end{equation}
It is easy to check that this functional $\|\cdot\|_{\star}$ satisfies the axioms of a norm; i.e., that for $f,g\in L^{p,\kappa}_{\rho,\infty}(\mathbb H^n)$ and $\lambda\in\mathbb R$,
\begin{itemize}
  \item it is positive definite: $\|f\|_{\star}\geq0$, and $\|f\|_{\star}=0\Leftrightarrow f=0$;
  \item it is multiplicative: $\|\lambda f\|_{\star}=|\lambda|\|f\|_{\star}$;
  \item it satisfies the triangle inequality: $\|f+g\|_{\star}\leq\|f\|_{\star}+\|g\|_{\star}$.
\end{itemize}
(iv) In view of \eqref{23}, for any given $f\in WL^{p,\kappa}_{\rho,\infty}(\mathbb H^n)$, let
\begin{equation*}
\theta^{**}:=\inf\big\{\theta>0:f\in WL^{p,\kappa}_{\rho,\theta}(\mathbb H^n)\big\}.
\end{equation*}
Similarly, we define the functional $\|\cdot\|_{\star\star}$ by
\begin{equation*}
\|f\|_{\star\star}=\big\|f\big\|_{WL^{p,\kappa}_{\rho,\infty}(\mathbb H^n)}:=\big\|f\big\|_{WL^{p,\kappa}_{\rho,\theta^{**}}(\mathbb H^n)}.
\end{equation*}
We can easily show that this functional $\|\cdot\|_{\star\star}$ satisfies the axioms of a (quasi)norm, and $WL^{p,\kappa}_{\rho,\infty}(\mathbb H^n)$ is a (quasi)normed linear space.
\end{rem}

Let $\mathcal I_{\alpha}$ be the operator as in \eqref{riesz}. Since Morrey space $L^{p,\kappa}_{\rho,\theta}(\mathbb H^n)$ (or weak Morrey space $WL^{p,\kappa}_{\rho,\theta}(\mathbb H^n)$) is a natural generalization of Lebesgue (or weak Lebesgue) space on $\mathbb H^n$ (when $\kappa=\theta=0$), it is accordingly natural to ask if the operator $\mathcal I_{\alpha}$ is also bounded in Morrey spaces. Our principal goal in this paper is to establish the Adams (Morrey-Sobolev) inequality on $\mathbb H^n$. We now state our main results as follows.

\begin{thm}\label{mainthm:1}
Let $0<\alpha<Q$, $1<p<Q/{\alpha}$, $0<\kappa<1-{(\alpha p)}/Q$ and $1/q=1/p-{\alpha}/{Q(1-\kappa)}$. If $V\in RH_s$ with $s\geq Q/2$, then the $\mathcal L$-fractional integral operator $\mathcal I_{\alpha}$ is bounded from $L^{p,\kappa}_{\rho,\infty}(\mathbb H^n)$ into $L^{q,\kappa}_{\rho,\infty}(\mathbb H^n)$.
\end{thm}

\begin{thm}\label{mainthm:2}
Let $0<\alpha<Q$, $p=1$, $0<\kappa<1-\alpha/Q$ and $1/q=1-{\alpha}/{Q(1-\kappa)}$. If $V\in RH_s$ with $s\geq Q/2$, then the $\mathcal L$-fractional integral operator $\mathcal I_{\alpha}$ is bounded from $L^{1,\kappa}_{\rho,\infty}(\mathbb H^n)$ into $WL^{q,\kappa}_{\rho,\infty}(\mathbb H^n)$.
\end{thm}

The second motivation of this paper is to treat the extreme cases $1>\kappa\geq 1-{(\alpha p)}/Q$. Before stating our next results, let us give some notation and definitions below. For given $0<\theta<\infty$, the space $\mathrm{BMO}_{\rho,\theta}(\mathbb H^n)$ is defined to be the set of all locally integrable functions $f$ on $\mathbb H^n$ satisfying
\begin{equation}\label{BM}
\frac{1}{|B(u_0,r)|}\int_{B(u_0,r)}\big|f(u)-f_{B(u_0,r)}\big|\,du
\leq C\cdot\left[1+\frac{r}{\rho(u_0)}\right]^{\theta},
\end{equation}
for all $u_0\in\mathbb H^n$ and $r>0$, $f_{B(u_0,r)}$ denotes the mean value of $f$ on $B(u_0,r)$, that is,
\begin{equation*}
f_{B(u_0,r)}:=\frac{1}{|B(u_0,r)|}\int_{B(u_0,r)}f(v)\,dv.
\end{equation*}
A norm for $f\in\mathrm{BMO}_{\rho,\theta}(\mathbb H^n)$, denoted by $\|f\|_{\mathrm{BMO}_{\rho,\theta}}$, is given by the infimum of the constants satisfying \eqref{BM}, or equivalently,
\begin{equation*}
\|f\|_{\mathrm{BMO}_{\rho,\theta}}
:=\sup_{B(u_0,r)}\left[1+\frac{r}{\rho(u_0)}\right]^{-\theta}\bigg(\frac{1}{|B(u_0,r)|}\int_{B(u_0,r)}\big|f(u)-f_{B(u_0,r)}\big|\,du\bigg),
\end{equation*}
where the supremum is taken over all balls $B(u_0,r)$ with $u_0\in\mathbb H^n$ and $r\in(0,\infty)$. If we identify functions that differ by a constant, then $\mathrm{BMO}_{\rho,\theta}(\mathbb H^n)$ becomes a Banach space with the norm $\|\cdot\|_{\mathrm{BMO}_{\rho,\theta}}$. We introduce a new space $\mathrm{BMO}_{\rho,\infty}(\mathbb H^n)$ defined by
\begin{equation*}
\mathrm{BMO}_{\rho,\infty}(\mathbb H^n):=\bigcup_{\theta>0}\mathrm{BMO}_{\rho,\theta}(\mathbb H^n).
\end{equation*}
Recall that in the setting of $\mathbb R^n$, the space $\mathrm{BMO}_{\rho,\theta}(\mathbb R^n)$ was first introduced by Bongioanni et al. \cite{bong3} (see also \cite{bong4}).

Moreover, for any given $\beta\in[0,1]$, we introduce the Campanato space on $\mathbb H^n$, with exponent $\beta$.
\begin{equation*}
\mathcal{C}^{\beta}_{\rho,\infty}(\mathbb H^n):=\bigcup_{\theta>0}\mathcal{C}^{\beta}_{\rho,\theta}(\mathbb H^n),
\end{equation*}
where for fixed $\theta\in(0,\infty)$ the space $\mathcal{C}^{\beta}_{\rho,\theta}(\mathbb H^n)$ is defined to be the set of all locally integrable functions $f$ satisfying
\begin{equation}\label{hconti}
\frac{1}{|B(u_0,r)|^{1+\beta/Q}}\int_{B(u_0,r)}\big|f(u)-f_{B(u_0,r)}\big|\,du
\leq C\cdot\left[1+\frac{r}{\rho(u_0)}\right]^{\theta},
\end{equation}
for all $u_0\in\mathbb H^n$ and $r\in(0,\infty)$. The smallest bound $C$ for which \eqref{hconti} is satisfied is then taken to be the norm of $f$ in this space and is denoted by $\|f\|_{\mathcal{C}^{\beta}_{\rho,\theta}}$. When $\theta=0$ or $V\equiv0$, $\mathrm{BMO}_{\rho,\theta}(\mathbb H^n)$ and $\mathcal{C}^{\beta}_{\rho,\theta}(\mathbb H^n)$ will be simply written as $\mathrm{BMO}(\mathbb H^n)$ (see, for example, \cite{folland3,krantz}) and $\mathcal{C}^{\beta}(\mathbb H^n)$, respectively.
\begin{rem}
(i) When $\beta=0$, this space $\mathcal{C}^{\beta}_{\rho,\theta}(\mathbb H^n)$ reduces to the space $\mathrm{BMO}_{\rho,\theta}(\mathbb H^n)$, and $\mathcal{C}^{\beta}_{\rho,\infty}(\mathbb H^n)$ reduces to the space $\mathrm{BMO}_{\rho,\infty}(\mathbb H^n)$ mentioned above.

(ii) We can also define a norm on the space $\mathcal{C}^{\beta}_{\rho,\infty}(\mathbb H^n)$ ($0\leq\beta\leq1$) by the same manner as in \eqref{newnorm}. The class $\mathcal{C}^{\beta}_{\rho,\infty}(\mathbb H^n)$ can be shown to be a Banach space under this norm.
\end{rem}

For the extreme cases $1>\kappa\geq 1-{(\alpha p)}/Q$ with $Q/{\alpha}>p\geq1$, we will prove the following result.
\begin{thm}\label{mainthm:3}
Let $0<\alpha<Q$, $1\leq p<Q/{\alpha}$ and $1-{(\alpha p)}/Q\leq\kappa<1$. If $V\in RH_s$ with $s\geq Q/2$, then the $\mathcal L$-fractional integral operator $\mathcal I_{\alpha}$ is bounded from $L^{p,\kappa}_{\rho,\infty}(\mathbb H^n)$ into $\mathcal{C}^{\beta}_{\rho,\infty}(\mathbb H^n)$ with $\beta/Q=\alpha/Q-{(1-\kappa)}/p$ and $\beta$ sufficiently small. To be more precise, $\beta<\delta\leq1$ and $\delta$ is given as in Lemma $\ref{kernel2}$ below.
\end{thm}

In particular, if we take $\kappa=1-{(\alpha p)}/Q$ (or $\beta=0$), then we obtain the following result on BMO-type estimate of $\mathcal I_{\alpha}$.
\begin{cor}\label{mainthm:4}
Let $0<\alpha<Q$, $1\leq p<Q/{\alpha}$ and $\alpha p=(1-\kappa)Q$. If $V\in RH_s$ with $s\geq Q/2$, then the $\mathcal L$-fractional integral operator $\mathcal I_{\alpha}$ is bounded from $L^{p,\kappa}_{\rho,\infty}(\mathbb H^n)$ into $\mathrm{BMO}_{\rho,\infty}(\mathbb H^n)$.
\end{cor}

\section{Proofs of Theorems $\ref{mainthm:1}$ and $\ref{mainthm:2}$}\label{sec3}
In this section, we will prove the conclusions of Theorems \ref{mainthm:1} and \ref{mainthm:2}. Assume that $P_s(u,v)$ is the kernel of the heat semigroup $\big\{e^{-s\mathcal L}\big\}_{s>0}$. Let us remind that the $\mathcal L$-fractional integral operator of order $\alpha\in(0,Q)$ can be written as
\begin{equation*}
\mathcal I_{\alpha}f(u)={\mathcal L}^{-{\alpha}/2}f(u)=\int_{\mathbb H^n}\mathcal K_{\alpha}(u,v)f(v)\,dv,
\end{equation*}
where
\begin{equation}\label{kauv}
\mathcal K_{\alpha}(u,v)=\frac{1}{\Gamma(\alpha/2)}\int_0^{\infty}P_s(u,v)\,s^{\alpha/2-1}ds.
\end{equation}

The following lemma gives the estimate of the kernel $\mathcal K_{\alpha}(u,v)$ related to the Schr\"odinger operator $\mathcal L$, which plays a key role in the proofs of our main theorems.
\begin{lem}\label{kernel}
Let $\rho$ be the auxiliary function as in \eqref{rho}. Let $V\in RH_s$ with $s\geq Q/2$ and $0<\alpha<Q$. For every positive integer $N\in\mathbb N$, there exists a positive constant $C_{N,\alpha}>0$ such that for all $u$ and $v$ in $\mathbb H^n$,
\begin{equation}\label{WH1}
\big|\mathcal K_{\alpha}(u,v)\big|\leq C_{N,\alpha}\bigg[1+\frac{|v^{-1}u|}{\rho(u)}\bigg]^{-N}\frac{1}{|v^{-1}u|^{Q-\alpha}}.
\end{equation}
\end{lem}
\begin{proof}
From Lemma \ref{ker1} and \eqref{kauv}, it follows that for $\alpha\in(0,Q)$,
\begin{equation*}
\begin{split}
\big|\mathcal K_{\alpha}(u,v)\big|&\leq\frac{1}{\Gamma(\alpha/2)}\int_0^{\infty}\big|P_s(u,v)\big|\,s^{\alpha/2-1}ds\\
&\leq\frac{1}{\Gamma(\alpha/2)}\int_0^{\infty}\frac{C_N}{s^{Q/2}}\cdot\exp\bigg(-\frac{|v^{-1}u|^2}{As}\bigg)
\bigg[1+\frac{\sqrt{s\,}}{\rho(u)}+\frac{\sqrt{s\,}}{\rho(v)}\bigg]^{-N}s^{\alpha/2-1}ds\\
&\leq\frac{1}{\Gamma(\alpha/2)}\int_0^{\infty}\frac{C_N}{s^{Q/2}}\cdot\exp\bigg(-\frac{|v^{-1}u|^2}{As}\bigg)
\bigg[1+\frac{\sqrt{s\,}}{\rho(u)}\bigg]^{-N}s^{\alpha/2-1}ds.
\end{split}
\end{equation*}
We now consider two cases $s>|v^{-1}u|^2$ and $0\leq s\leq|v^{-1}u|^2$, respectively. Thus, $|\mathcal K_{\alpha}(u,v)|\leq I+II$, where
\begin{equation*}
I=\frac{1}{\Gamma(\alpha/2)}\int_{|v^{-1}u|^2}^{\infty}\frac{C_N}{s^{Q/2}}\cdot\exp\bigg(-\frac{|v^{-1}u|^2}{As}\bigg)
\bigg[1+\frac{\sqrt{s\,}}{\rho(u)}\bigg]^{-N}s^{\alpha/2-1}ds
\end{equation*}
and
\begin{equation*}
II=\frac{1}{\Gamma(\alpha/2)}\int_0^{|v^{-1}u|^2}\frac{C_N}{s^{Q/2}}\cdot\exp\bigg(-\frac{|v^{-1}u|^2}{As}\bigg)
\bigg[1+\frac{\sqrt{s\,}}{\rho(u)}\bigg]^{-N}s^{\alpha/2-1}ds.
\end{equation*}
When $s>|v^{-1}u|^2$, then $\sqrt{s\,}>|v^{-1}u|$, and hence
\begin{equation*}
\begin{split}
I&\leq\frac{1}{\Gamma(\alpha/2)}\int_{|v^{-1}u|^2}^{\infty}\frac{C_N}{s^{Q/2}}\cdot\exp\bigg(-\frac{|v^{-1}u|^2}{As}\bigg)
\bigg[1+\frac{|v^{-1}u|}{\rho(u)}\bigg]^{-N}s^{\alpha/2-1}ds\\
&\leq C_{N,\alpha}\bigg[1+\frac{|v^{-1}u|}{\rho(u)}\bigg]^{-N}\int_{|v^{-1}u|^2}^{\infty}s^{\alpha/2-Q/2-1}ds\\
&\leq C_{N,\alpha}\bigg[1+\frac{|v^{-1}u|}{\rho(u)}\bigg]^{-N}\frac{1}{|v^{-1}u|^{Q-\alpha}},
\end{split}
\end{equation*}
where the last integral converges because $0<\alpha<Q$. On the other hand,
\begin{equation*}
\begin{split}
II&\leq C_{N,\alpha}\int_0^{|v^{-1}u|^2}\frac{1}{s^{Q/2}}\cdot\bigg(\frac{|v^{-1}u|^2}{s}\bigg)^{-(Q/2+N/2)}
\bigg[1+\frac{\sqrt{s\,}}{\rho(u)}\bigg]^{-N}s^{\alpha/2-1}ds\\
&=C_{N,\alpha}\int_0^{|v^{-1}u|^2}\frac{1}{|v^{-1}u|^Q}\cdot\bigg(\frac{\sqrt{s\,}}{|v^{-1}u|}\bigg)^{N}
\bigg[1+\frac{\sqrt{s\,}}{\rho(u)}\bigg]^{-N}s^{\alpha/2-1}ds.
\end{split}
\end{equation*}
It is easy to see that if $0\leq s\leq|v^{-1}u|^2$, then it holds true that
\begin{equation*}
\frac{\sqrt{s\,}}{|v^{-1}u|}\leq\frac{\sqrt{s\,}+\rho(u)}{|v^{-1}u|+\rho(u)}.
\end{equation*}
Hence,
\begin{equation*}
\begin{split}
II&\leq C_{N,\alpha}\int_0^{|v^{-1}u|^2}\frac{1}{|v^{-1}u|^Q}\cdot\bigg[\frac{\sqrt{s\,}+\rho(u)}{|v^{-1}u|+\rho(u)}\bigg]^{N}
\bigg[\frac{\sqrt{s\,}+\rho(u)}{\rho(u)}\bigg]^{-N}s^{\alpha/2-1}ds\\
&=\frac{C_{N,\alpha}}{|v^{-1}u|^Q}\bigg[1+\frac{|v^{-1}u|}{\rho(u)}\bigg]^{-N}\int_0^{|v^{-1}u|^2}s^{\alpha/2-1}ds\\
&=C_{N,\alpha}\bigg[1+\frac{|v^{-1}u|}{\rho(u)}\bigg]^{-N}\frac{1}{|v^{-1}u|^{Q-\alpha}}.
\end{split}
\end{equation*}
Combining the estimates of $I$ and $II$ yields the desired estimate \eqref{WH1} for $\alpha\in(0,Q)$. This concludes the proof of the lemma.
\end{proof}

Let $\rho$ be the critical radius function as in \eqref{rho}. For fixed $N\in\mathbb N$, the maximal operator $\mathcal M_{\rho,N}$ is defined by setting, for any $f\in L^1_{\mathrm{loc}}(\mathbb H^n)$ and $u\in \mathbb H^n$,
\begin{equation*}
\mathcal M_{\rho,N}f(u):=\sup_{r>0}\left[1+\frac{r}{\rho(u)}\right]^{-N}\frac{1}{|B(u,r)|}\int_{B(u,r)}|f(v)|\,dv,
\end{equation*}
where the supremum is taken over all balls $B$ in $\mathbb H^n$ centered at $u$. The following lemma is very useful to us, which gives the relation between $\mathcal I_{\alpha}$ and maximal operator $\mathcal M_{\rho,N}$ on the Heisenberg group.
\begin{lem}\label{keylem}
Let $\rho$ be as in \eqref{rho}. Let $0<\alpha<Q$, $1\leq p<Q/{\alpha}$, $0<\kappa<1-{(\alpha p)}/Q$ and $1/q=1/p-\alpha/{Q(1-\kappa)}$. Then for any positive integer $N\in\mathbb N$ and any $f\in L^{p,\kappa}_{\rho,\theta}(\mathbb H^n)$ with $\theta\leq N$, we have
\begin{equation*}
|\mathcal I_{\alpha}f(u)|\lesssim\Big[\mathcal M_{\rho,N}f(u)\Big]^{p/q}\cdot\Big[\big\|f\big\|_{L^{p,\kappa}_{\rho,\theta}(\mathbb H^n)}\Big]^{1-p/q},\quad u\in\mathbb H^n.
\end{equation*}
\end{lem}
\begin{proof}[Proof of Lemma $\ref{keylem}$]
Following the idea of \cite{adams} (see also \cite{adams1}), by definition and Lemma \ref{kernel}, we decompose the function $\mathcal I_{\alpha}f$ as
\begin{equation*}
\begin{split}
|\mathcal I_{\alpha}f(u)|&=\bigg|\int_{\mathbb H^n}\mathcal K_{\alpha}(u,v)f(v)\,dv\bigg|\\
&\leq C_{N,\alpha}\int_{\mathbb H^n}\bigg[1+\frac{|v^{-1}u|}{\rho(u)}\bigg]^{-N}\frac{1}{|v^{-1}u|^{Q-\alpha}}\cdot|f(v)|\,dv\\
&=C_{N,\alpha}\bigg(\int_{|v^{-1}u|<\sigma}\cdots+\int_{|v^{-1}u|\geq\sigma}\cdots\bigg):=\mathcal I_{\alpha}^{(1)}f(u)+\mathcal I_{\alpha}^{(2)}f(u),
\end{split}
\end{equation*}
where $\sigma>0$ is a constant which will be determined later. In what follows, we consider each part separately. For the first integral, we have
\begin{equation*}
\begin{split}
\mathcal I_{\alpha}^{(1)}f(u)
&=C_{N,\alpha}\sum_{j=1}^\infty\int_{2^{-j}\sigma\leq|v^{-1}u|<2^{-j+1}\sigma}\bigg[1+\frac{|v^{-1}u|}{\rho(u)}\bigg]^{-N}\frac{1}{|v^{-1}u|^{Q-\alpha}}\cdot|f(v)|\,dv\\
&\leq C_{N,\alpha}\sum_{j=1}^\infty\int_{2^{-j}\sigma\leq|v^{-1}u|<2^{-j+1}\sigma}\frac{1}{(2^{-j}\sigma)^{Q-\alpha}}\bigg[1+\frac{2^{-j}\sigma}{\rho(u)}\bigg]^{-N}\cdot|f(v)|\,dv\\
&\leq C_{N,\alpha,n}\sum_{j=1}^\infty\frac{(2^{-j+1}\sigma)^Q}{(2^{-j}\sigma)^{Q-\alpha}}\\
&\times\bigg[1+\frac{2^{-j+1}\sigma}{\rho(u)}\bigg]^{-N}\frac{1}{|B(u,2^{-j+1}\sigma)|}\int_{|v^{-1}u|<2^{-j+1}\sigma}|f(v)|\,dv.
\end{split}
\end{equation*}
By the definition of the maximal operator $\mathcal M_{\rho,N}$, we thus obtain
\begin{equation*}
\begin{split}
\mathcal I_{\alpha}^{(1)}f(u)
&\leq C_{N,\alpha,n}\sum_{j=1}^\infty\frac{\sigma^\alpha}{2^{j\alpha}}\cdot \mathcal M_{\rho,N}f(u)
 \leq C\sigma^\alpha\cdot\mathcal M_{\rho,N}f(u).
\end{split}
\end{equation*}
For the second integral, we have
\begin{equation*}
\begin{split}
\mathcal I_{\alpha}^{(2)}f(u)
&=C_{N,\alpha}\sum_{j=1}^\infty\int_{2^{j-1}\sigma\leq|v^{-1}u|<2^{j}\sigma}\bigg[1+\frac{|v^{-1}u|}{\rho(u)}\bigg]^{-N}\frac{1}{|v^{-1}u|^{Q-\alpha}}\cdot|f(v)|\,dv\\
&\leq C_{N,\alpha}\sum_{j=1}^\infty\int_{2^{j-1}\sigma\leq|v^{-1}u|<2^{j}\sigma}\frac{1}{(2^{j-1}\sigma)^{Q-\alpha}}\bigg[1+\frac{2^{j-1}\sigma}{\rho(u)}\bigg]^{-N}\cdot|f(v)|\,dv.
\end{split}
\end{equation*}
Let $p'$ be the conjugate exponent of $p$ and $1'=\infty$. By using H\"older's inequality, we obtain that
\begin{equation*}
\begin{split}
\mathcal I_{\alpha}^{(2)}f(u)&\leq C_{N,\alpha}\sum_{j=1}^\infty\frac{1}{(2^{j-1}\sigma)^{Q-\alpha}}\bigg[1+\frac{2^{j-1}\sigma}{\rho(u)}\bigg]^{-N}\\
&\times\bigg(\int_{|v^{-1}u|<2^{j}\sigma}|f(v)|^p\,dv\bigg)^{1/p}
\bigg(\int_{|v^{-1}u|<2^{j}\sigma}1\,dv\bigg)^{1/{p'}}\\
&\leq C_{N,\alpha}\sum_{j=1}^\infty\frac{1}{(2^{j}\sigma)^{Q-\alpha}}\bigg[1+\frac{2^{j}\sigma}{\rho(u)}\bigg]^{-N+\theta}
\big|B(u,2^{j}\sigma)\big|^{\kappa/p+1/{p'}}\big\|f\big\|_{L^{p,\kappa}_{\rho,\theta}(\mathbb H^n)}.
\end{split}
\end{equation*}
By the assumption $\theta\leq N$, we can further obtain
\begin{equation*}
\begin{split}
\mathcal I_{\alpha}^{(2)}f(u)&\leq C_{N,\alpha,n}\big\|f\big\|_{L^{p,\kappa}_{\rho,\theta}(\mathbb H^n)}
\sum_{j=1}^\infty\frac{(2^{j}\sigma)^{Q[\kappa/p+1/{p'}]}}{(2^{j}\sigma)^{Q-\alpha}}\\
&=C\big\|f\big\|_{L^{p,\kappa}_{\rho,\theta}(\mathbb H^n)}
\sum_{j=1}^\infty\frac{1}{(2^{j}\sigma)^{Q[1/p-\kappa/p-\alpha/Q]}}.
\end{split}
\end{equation*}
Note that ${(1-\kappa)}/p>\alpha/Q$ and the last series is convergent. Then we have
\begin{equation*}
\begin{split}
\mathcal I_{\alpha}^{(2)}f(u)&\leq C\sigma^{\alpha+Q(\kappa-1)/p}\cdot\big\|f\big\|_{L^{p,\kappa}_{\rho,\theta}(\mathbb H^n)}.
\end{split}
\end{equation*}
Summing up the above estimates for $\mathcal I_{\alpha}^{(1)}f$ and $\mathcal I_{\alpha}^{(2)}f$, we have
\begin{equation}\label{twosum}
\begin{split}
|\mathcal I_{\alpha}f(u)|&\lesssim\Big[\sigma^\alpha\cdot\mathcal M_{\rho,N}f(u)+\sigma^{\alpha+Q(\kappa-1)/p}\cdot\big\|f\big\|_{L^{p,\kappa}_{\rho,\theta}(\mathbb H^n)}\Big].
\end{split}
\end{equation}
We now choose $\sigma$ such that
\begin{equation*}
\sigma^\alpha\cdot\mathcal M_{\rho,N}f(u)=\sigma^{\alpha+Q(\kappa-1)/p}\cdot\big\|f\big\|_{L^{p,\kappa}_{\rho,\theta}(\mathbb H^n)}.
\end{equation*}
That is,
\begin{equation*}
\sigma^{Q(1-\kappa)/p}=\frac{\big\|f\big\|_{L^{p,\kappa}_{\rho,\theta}(\mathbb H^n)}}{\mathcal M_{\rho,N}f(u)}.
\end{equation*}
Putting this back into \eqref{twosum} and noting that $p/{Q(1-\kappa)}\cdot\alpha=(1/p-1/q)\cdot p$, we deduce that for any $u\in\mathbb H^n$,
\begin{equation*}
\begin{split}
|\mathcal I_{\alpha}f(u)|\lesssim\sigma^\alpha\cdot\mathcal M_{\rho,N}f(u)
&=\Bigg[\frac{\big\|f\big\|_{L^{p,\kappa}_{\rho,\theta}(\mathbb H^n)}}{\mathcal M_{\rho,N}f(u)}\Bigg]^{(1/p-1/q)\cdot p}\cdot\mathcal M_{\rho,N}f(u)\\
&=\Big[\mathcal M_{\rho,N}f(u)\Big]^{p/q}\cdot\Big[\big\|f\big\|_{L^{p,\kappa}_{\rho,\theta}(\mathbb H^n)}\Big]^{1-p/q},
\end{split}
\end{equation*}
which is the desired conclusion.
\end{proof}

Furthermore, let us give the following two basic estimates for the maximal operator $\mathcal M_{\rho,N}$ on the Morrey spaces $L^{p,\kappa}_{\rho,\theta}(\mathbb H^n)$ with $\theta\in(0,N]$ and $(p,\kappa)\in[1,\infty)\times(0,1)$.
\begin{thm}\label{maxthm:1}
Let $1<p<\infty$ and $0<\kappa<1$. If $V\in RH_s$ with $s\geq Q/2$, then for fixed $N\in\mathbb N$, the maximal operator $\mathcal M_{\rho,N}$ is bounded on $L^{p,\kappa}_{\rho,\theta}(\mathbb H^n)$ for every $\theta$ with $0<\theta(N_0+1)\leq N$. Here $N_0$ is the same as in \eqref{com}.
\end{thm}

\begin{thm}\label{maxthm:2}
Let $p=1$ and $0<\kappa<1$. If $V\in RH_s$ with $s\geq Q/2$, then for fixed $N\in\mathbb N$, the maximal operator $\mathcal M_{\rho,N}$ is bounded from $L^{1,\kappa}_{\rho,\theta}(\mathbb H^n)$ into $WL^{1,\kappa}_{\rho,\theta}(\mathbb H^n)$ for every $\theta$ with $0<\theta(N_0+1)\leq N$. Here $N_0$ is the same as in \eqref{com}.
\end{thm}

\begin{proof}[Proof of Theorem $\ref{maxthm:1}$]
The Hardy-Littlewood maximal operator $\mathcal M$ is defined by setting, for any $f\in L^1_{\mathrm{loc}}(\mathbb H^n)$ and $u\in \mathbb H^n$,
\begin{equation*}
\mathcal Mf(u):=\sup_{u\in B}\frac{1}{|B|}\int_{B}|f(v)|\,dv,
\end{equation*}
where the supremum is taken over all balls $B$ in $\mathbb H^n$ containing $u$. It is shown that the maximal operator $\mathcal M$ on $\mathbb H^n$ is of strong-type $(p,p)$ for $1<p<\infty$, and is of weak-type $(1,1)$. This result can be found in \cite{folland2,lin}.

Obviously, for any $f\in L^1_{\mathrm{loc}}(\mathbb H^n)$ and for fixed $N\in\mathbb N$, the maximal function $\mathcal M_{\rho,N}(f)$ is dominated by $\mathcal M(f)$, which implies that the maximal operator $\mathcal M_{\rho,N}$ on $\mathbb H^n$ is also of strong-type $(p,p)$ for $1<p<\infty$, and is of weak-type $(1,1)$. By definition, we only need to show that for any given ball $B=B(u_0,r)$ of $\mathbb H^n$, there is a constant $C>0$ such that
\begin{equation}\label{Main1}
\bigg[\frac{1}{|B|^{\kappa}}\int_B\big|\mathcal M_{\rho,N}f(u)\big|^p\,du\bigg]^{1/p}\leq C\cdot\left[1+\frac{r}{\rho(u_0)}\right]^{\theta}
\end{equation}
holds for given $f\in L^{p,\kappa}_{\rho,\theta}(\mathbb H^n)$ with $(p,\kappa)\in(1,\infty)\times(0,1)$ and $\theta(N_0+1)\leq N$. Using the standard technique, we decompose the function $f$ as
\begin{equation*}
\begin{cases}
f=f_1+f_2\in L^{p,\kappa}_{\rho,\theta}(\mathbb H^n);\  &\\
f_1=f\cdot\chi_{2B};\  &\\
f_2=f\cdot\chi_{(2B)^{\complement}},
\end{cases}
\end{equation*}
where $2B$ is the ball centered at $u_0$ of radius $2r>0$, $\chi_{2B}$ is the characteristic function of $2B$ and $(2B)^{\complement}=\mathbb H^n\backslash(2B)$. Then by the sublinearity of $\mathcal M_{\rho,N}$, we write
\begin{equation*}
\begin{split}
\bigg[\frac{1}{|B|^{\kappa}}\int_B\big|\mathcal M_{\rho,N}f(u)\big|^p\,du\bigg]^{1/p}
&\leq\bigg[\frac{1}{|B|^{\kappa}}\int_B\big|\mathcal M_{\rho,N}f_1(u)\big|^p\,du\bigg]^{1/p}\\
&+\bigg[\frac{1}{|B|^{\kappa}}\int_B\big|\mathcal M_{\rho,N}f_2(u)\big|^p\,du\bigg]^{1/p}:=I_1+I_2.
\end{split}
\end{equation*}
In what follows, we consider each part separately. For the first term $I_1$, by the $(L^p,L^p)$-boundedness of $\mathcal M_{\rho,N}$ for $1<p<\infty$, we have
\begin{equation*}
\begin{split}
I_1&=\bigg[\frac{1}{|B|^{\kappa}}\int_B\big|\mathcal M_{\rho,N}f_1(u)\big|^p\,du\bigg]^{1/p}\\
&\leq C\cdot\frac{1}{|B|^{\kappa/p}}\bigg[\int_{\mathbb H^n}\big|f_1(u)\big|^p\,du\bigg]^{1/p}\\
&=C\cdot\frac{1}{|B|^{\kappa/p}}\bigg[\int_{2B}\big|f(u)\big|^p\,du\bigg]^{1/p}\\
&\leq C\big\|f\big\|_{L^{p,\kappa}_{\rho,\theta}(\mathbb H^n)}\cdot
\frac{|2B|^{\kappa/p}}{|B|^{\kappa/p}}\cdot\left[1+\frac{2r}{\rho(u_0)}\right]^{\theta}.
\end{split}
\end{equation*}
Also observe that for any given $\theta>0$,
\begin{equation}\label{2rx}
1\leq\left[1+\frac{2r}{\rho(u_0)}\right]^{\theta}\leq 2^{\theta}\left[1+\frac{r}{\rho(u_0)}\right]^{\theta}.
\end{equation}
This in turn implies that
\begin{equation*}
\begin{split}
I_1&\leq C_{\theta,n}\big\|f\big\|_{L^{p,\kappa}_{\rho,\theta}(\mathbb H^n)}\left[1+\frac{r}{\rho(u_0)}\right]^{\theta}.
\end{split}
\end{equation*}
Next we estimate the other term $I_2$. By a simple calculation, we know that if $B(u,r')\cap B(u_0,2r)^{\complement}\neq\O$ and $u\in B(u_0,r)$, then
\begin{equation*}
r'>r\quad  \mathrm{and} \quad B(u,r')\cap B(u_0,2r)^{\complement}\subset B(u_0,2r').
\end{equation*}
Indeed, for any $v\in B(u,r')\cap B(u_0,2r)^{\complement}$, one has
\begin{equation*}
\begin{split}
r'&>\big|v^{-1}u\big|=\big|(v^{-1}u_0)\cdot(u_0^{-1}u)\big|\\
&\geq\big|v^{-1}u_0\big|-\big|u_0^{-1}u\big|>2r-r=r.
\end{split}
\end{equation*}
Moreover, for any $v\in B(u,r')\cap B(u_0,2r)^{\complement}$,
\begin{equation*}
\begin{split}
\big|v^{-1}u_0\big|&=\big|(v^{-1}u)\cdot(u^{-1}u_0)\big|\\
&\leq\big|v^{-1}u\big|+\big|u^{-1}u_0\big|<r'+r<2r'.
\end{split}
\end{equation*}
Thus, for any $u\in B(u_0,r)$ and $p\in[1,\infty)$, it follows from H\"older's inequality that
\begin{equation*}
\begin{split}
\mathcal M_{\rho,N}f_2(u)&=\sup_{r'>0}\left[1+\frac{r'}{\rho(u)}\right]^{-N}\frac{1}{|B(u,r')|}\int_{B(u,r')\cap B(u_0,2r)^{\texttt{C}}}|f(v)|\,dv\\
&\leq\sup_{r'>r}\left[1+\frac{r'}{\rho(u)}\right]^{-N}\frac{1}{|B(u,r')|}\int_{B(u_0,2r')}|f(v)|\,dv\\
&\leq\sup_{r'>r}\left[1+\frac{r'}{\rho(u)}\right]^{-N}\frac{2^Q}{|B(u_0,2r')|}\bigg(\int_{B(u_0,2r')}\big|f(v)\big|^p\,dv\bigg)^{1/p}
\bigg(\int_{B(u_0,2r')}1\,dv\bigg)^{1/{p'}}\\
&\leq C\big\|f\big\|_{L^{p,\kappa}_{\rho,\theta}(\mathbb H^n)}\sup_{r'>r}\frac{1}{|B(u_0,2r')|^{{(1-\kappa)}/p}}
\left[1+\frac{r'}{\rho(u)}\right]^{-N}\left[1+\frac{2r'}{\rho(u_0)}\right]^{\theta}.
\end{split}
\end{equation*}
Note that $r'>r$. Then we have $u\in B(u_0,r)\subset B(u_0,r')$. In view of \eqref{com2} with $k=0$ and \eqref{2rx}, we can further obtain
\begin{equation*}
\big|\mathcal M_{\rho,N}f_2(u)\big|
\leq C_{n,\theta}\big\|f\big\|_{L^{p,\kappa}_{\rho,\theta}(\mathbb H^n)}
\sup_{r'>r}\frac{1}{|B(u_0,2r')|^{{(1-\kappa)}/p}}
\left[1+\frac{r'}{\rho(u_0)}\right]^{(-N)\cdot\frac{1}{N_0+1}+\theta}.
\end{equation*}
Thus, by choosing $N$ large enough such that $N\geq\theta(N_0+1)$ and noting that $1-\kappa>0$, then we get
\begin{equation}\label{Mf2}
\big|\mathcal M_{\rho,N}f_2(u)\big|
\leq C_{n,\theta}\big\|f\big\|_{L^{p,\kappa}_{\rho,\theta}(\mathbb H^n)}
\frac{1}{|B(u_0,2r)|^{{(1-\kappa)}/p}}\left[1+\frac{r}{\rho(u_0)}\right]^{(-N)\cdot\frac{1}{N_0+1}+\theta}.
\end{equation}
Therefore, from this pointwise estimate for $\mathcal M_{\rho,N}(f_2)$, it follows that
\begin{equation*}
\begin{split}
I_2&\leq C\big\|f\big\|_{L^{p,\kappa}_{\rho,\theta}(\mathbb H^n)}
\left[1+\frac{r}{\rho(u_0)}\right]^{(-N)\cdot\frac{1}{N_0+1}+\theta}\\
&\leq C\big\|f\big\|_{L^{p,\kappa}_{\rho,\theta}(\mathbb H^n)}
\left[1+\frac{r}{\rho(u_0)}\right]^{\theta}.
\end{split}
\end{equation*}
Summing up the above estimates for $I_1$ and $I_2$, we obtain the desired inequality \eqref{Main1}.
\end{proof}

\begin{proof}[Proof of Theorem $\ref{maxthm:2}$]
To prove Theorem \ref{maxthm:2}, by definition, it suffices to prove that for each given ball $B=B(u_0,r)$ of $\mathbb H^n$, the inequality
\begin{equation}\label{Main2}
\frac{1}{|B|^{\kappa}}\sup_{\lambda>0}\lambda\cdot\big|\big\{u\in B:|\mathcal M_{\rho,N}f(u)|>\lambda\big\}\big|
\leq C\cdot\left[1+\frac{r}{\rho(u_0)}\right]^{\theta}
\end{equation}
holds for any given $f\in L^{1,\kappa}_{\rho,\theta}(\mathbb H^n)$ with $0<\kappa<1$ and $\theta(N_0+1)\leq N$. We decompose the function $f$ as
\begin{equation*}
\begin{cases}
f=f_1+f_2\in L^{1,\kappa}_{\rho,\theta}(\mathbb H^n);\  &\\
f_1=f\cdot\chi_{2B};\  &\\
f_2=f\cdot\chi_{(2B)^{\complement}}.
\end{cases}
\end{equation*}
Then for any given $\lambda>0$, by the sublinearity of $\mathcal M_{\rho,N}$, we can write
\begin{equation*}
\begin{split}
&\frac{1}{|B|^{\kappa}}\lambda\cdot\big|\big\{u\in B:|\mathcal M_{\rho,N}f(u)|>\lambda\big\}\big|\\
&\leq\frac{1}{|B|^{\kappa}}\lambda\cdot\big|\big\{u\in B:|\mathcal M_{\rho,N}f_1(u)|>\lambda/2\big\}\big|\\
&+\frac{1}{|B|^{\kappa}}\lambda\cdot\big|\big\{u\in B:|\mathcal M_{\rho,N}f_2(u)|>\lambda/2\big\}\big|:=J_1+J_2.
\end{split}
\end{equation*}
We first give the estimate for the term $J_1$. By the $(L^1,WL^1)$-boundedness of $\mathcal M_{\rho,N}$, we get
\begin{equation*}
\begin{split}
J_1&=\frac{1}{|B|^{\kappa}}\lambda\cdot\big|\big\{u\in B:|\mathcal M_{\rho,N}f_1(u)|>\lambda/2\big\}\big|\\
&\leq C\cdot\frac{1}{|B|^{\kappa}}\bigg(\int_{\mathbb H^n}\big|f_1(u)\big|\,du\bigg)\\
&=C\cdot\frac{1}{|B|^{\kappa}}\bigg(\int_{2B}\big|f(u)\big|\,du\bigg)\\
&\leq C\big\|f\big\|_{L^{1,\kappa}_{\rho,\theta}(\mathbb H^n)}\cdot\frac{|2B|^{\kappa}}{|B|^{\kappa}}\left[1+\frac{2r}{\rho(u_0)}\right]^{\theta}.
\end{split}
\end{equation*}
Therefore, in view of \eqref{2rx},
\begin{equation*}
J_1\leq C\big\|f\big\|_{L^{1,\kappa}_{\rho,\theta}(\mathbb H^n)}\cdot\left[1+\frac{r}{\rho(u_0)}\right]^{\theta}.
\end{equation*}
As for the second term $J_2$, by using the pointwise inequality \eqref{Mf2} and Chebyshev's inequality, we can deduce that
\begin{equation}\label{Tf2pr}
\begin{split}
J_2&=\frac{1}{|B|^{\kappa}}\lambda\cdot\big|\big\{u\in B:|\mathcal M_{\rho,N}f_2(u)|>\lambda/2\big\}\big|\leq\frac{2}{|B|^{\kappa}}\bigg(\int_{B}\big|\mathcal M_{\rho,N}f_2(u)\big|\,du\bigg)\\
&\leq C\cdot\frac{|B(u_0,r)|}{|B(u_0,r)|^{\kappa}}
\big\|f\big\|_{L^{1,\kappa}_{\rho,\theta}(\mathbb H^n)}\cdot\frac{1}{|B(u_0,2r)|^{{(1-\kappa)}}}
\left[1+\frac{r}{\rho(u_0)}\right]^{(-N)\cdot\frac{1}{N_0+1}+\theta}\\
&\leq C\big\|f\big\|_{L^{1,\kappa}_{\rho,\theta}(\mathbb H^n)}
\left[1+\frac{r}{\rho(u_0)}\right]^{\theta}.
\end{split}
\end{equation}
Summing up the above estimates for $J_1$ and $J_2$, and then taking the supremum over all $\lambda>0$, we obtain the desired inequality \eqref{Main2}. This finishes the proof of Theorem \ref{maxthm:2}.
\end{proof}

Base on Lemma \ref{keylem} and Theorems \ref{maxthm:1} and \ref{maxthm:2}, we are now in a position to prove our main results.
\begin{proof}[Proofs of Theorems $\ref{mainthm:1}$ and $\ref{mainthm:2}$]
For any given $f\in L^{p,\kappa}_{\rho,\infty}(\mathbb H^n)$ with $1\leq p<Q/{\alpha}$ and $0<\kappa<1-{(\alpha p)}/Q$, suppose that $f\in L^{p,\kappa}_{\rho,\theta^*}(\mathbb H^n)$, where
\begin{equation*}
\theta^*=\inf\big\{\theta>0:f\in L^{p,\kappa}_{\rho,\theta}(\mathbb H^n)\big\}\quad\mathrm{and}\quad\big\|f\big\|_{L^{p,\kappa}_{\rho,\infty}(\mathbb H^n)}=\big\|f\big\|_{L^{p,\kappa}_{\rho,\theta^*}(\mathbb H^n)}.
\end{equation*}
Let $N\in\mathbb N$ be large enough such that $N\geq\theta^*(N_0+1)\geq\theta^*$ and $N_0$ is as in \eqref{com}.
(1) For $p>1$, by Lemma \ref{keylem} and Theorem \ref{maxthm:1}, we can deduce that for each given ball $B(u_0,r)$ in $\mathbb H^n$,
\begin{equation*}
\begin{split}
&\bigg[1+\frac{r}{\rho(u_0)}\bigg]^{-\theta^*}
\bigg(\frac{1}{|B(u_0,r)|^{\kappa}}\int_{B(u_0,r)}\big|\mathcal I_{\alpha}f(u)\big|^q\,du\bigg)^{1/q}\\
&\leq C\Big[\big\|f\big\|_{L^{p,\kappa}_{\rho,\infty}(\mathbb H^n)}\Big]^{1-p/q}\bigg[1+\frac{r}{\rho(u_0)}\bigg]^{-\theta^*}
\bigg(\frac{1}{|B(u_0,r)|^{\kappa}}\int_{B(u_0,r)}\big|\mathcal M_{\rho,N}f(u)\big|^p\,du\bigg)^{1/q}\\
&\leq C\Big[\big\|f\big\|_{L^{p,\kappa}_{\rho,\infty}(\mathbb H^n)}\Big]^{1-p/q}\cdot\Big[\big\|f\big\|_{L^{p,\kappa}_{\rho,\infty}(\mathbb H^n)}\Big]^{p/q}
=C\big\|f\big\|_{L^{p,\kappa}_{\rho,\infty}(\mathbb H^n)}.
\end{split}
\end{equation*}
(2) On the other hand, for $p=1$, by Lemma \ref{keylem}, we can deduce that for each given ball $B(u_0,r)$ in $\mathbb H^n$,
\begin{equation*}
\begin{split}
&\left[1+\frac{r}{\rho(u_0)}\right]^{-\theta^*}\frac{1}{|B(u_0,r)|^{\kappa/q}}
\sup_{\lambda>0}\lambda\cdot\big|\big\{u\in B(u_0,r):|\mathcal I_{\alpha}f(u)|>\lambda\big\}\big|^{1/q}\\
&\lesssim\left[1+\frac{r}{\rho(u_0)}\right]^{-\theta^*}\frac{1}{|B(u_0,r)|^{\kappa/q}}\\
&\times\sup_{\lambda>0}\lambda\cdot\bigg|\bigg\{u\in B(u_0,r):\Big[\mathcal M_{\rho,N}f(u)\Big]^{1/q}\cdot\Big[\big\|f\big\|_{L^{1,\kappa}_{\rho,\infty}(\mathbb H^n)}\Big]^{1-1/q}>\lambda\bigg\}\bigg|^{1/q}\\
&=\left[1+\frac{r}{\rho(u_0)}\right]^{-\theta^*}\frac{1}{|B(u_0,r)|^{\kappa/q}}
\times\sup_{\lambda>0}\lambda\cdot\big|\big\{u\in B(u_0,r):\mathcal M_{\rho,N}f(u)>\sigma\big\}\big|^{1/q},
\end{split}
\end{equation*}
where
\begin{equation*}
\sigma:=\frac{\lambda^q}{\big[\big\|f\big\|_{L^{1,\kappa}_{\rho,\infty}(\mathbb H^n)}\big]^{q-1}}.
\end{equation*}
Furthermore, by using Theorem \ref{maxthm:2}, we have
\begin{equation*}
\begin{split}
&\left[1+\frac{r}{\rho(u_0)}\right]^{-\theta^*}\frac{1}{|B(u_0,r)|^{\kappa/q}}
\sup_{\lambda>0}\lambda\cdot\big|\big\{u\in B(u_0,r):|\mathcal I_{\alpha}f(u)|>\lambda\big\}\big|^{1/q}\\
&\leq C\frac{1}{|B(u_0,r)|^{\kappa/q}}\sup_{\lambda>0}\lambda\cdot\bigg(\frac{|B(u_0,r)|^\kappa\big\|f\big\|_{L^{1,\kappa}_{\rho,\infty}(\mathbb H^n)}}{\sigma}\bigg)^{1/q}\\
&=C\Big[\big\|f\big\|_{L^{1,\kappa}_{\rho,\infty}(\mathbb H^n)}\Big]^{1-1/q}\cdot\Big[\big\|f\big\|_{L^{1,\kappa}_{\rho,\infty}(\mathbb H^n)}\Big]^{1/q}
=C\big\|f\big\|_{L^{1,\kappa}_{\rho,\infty}(\mathbb H^n)}.
\end{split}
\end{equation*}
Finally, by taking the supremum over all balls $B(u_0,r)$ in $\mathbb H^n$, we conclude the proof.
\end{proof}

\section{Proof of Theorem \ref{mainthm:3}}\label{sec4}
Assume that $\big\{e^{-s\mathcal L}\big\}_{s>0}$ is the heat semigroup generated by $\mathcal L$ and $P_s(u,v)$ is the kernel of the semigroup $\big\{e^{-s\mathcal L}\big\}_{s>0}$. We need the following lemma which establishes the Lipschitz regularity of the kernel $P_s(u,v)$. See Lemma 11 and Remark 4 in \cite{lin}.
\begin{lem}[\cite{lin}]\label{ker2}
Let $\rho$ be as in \eqref{rho}. Let $V\in RH_s$ with $s\geq Q/2$. For every positive integer $N\in\mathbb N$, there exists a positive constant $C_N>0$ such that for all $u$ and $v$ in $\mathbb H^n$, and for some fixed $0<\delta\leq 1$,
\begin{equation*}
\big|P_s(u\cdot h,v)-P_s(u,v)\big|\leq C_N\bigg(\frac{|h|}{\sqrt{s\,}}\bigg)^{\delta} s^{-Q/2}\exp\bigg(-\frac{|v^{-1}u|^2}{As}\bigg)\bigg[1+\frac{\sqrt{s\,}}{\rho(u)}+\frac{\sqrt{s\,}}{\rho(v)}\bigg]^{-N},
\end{equation*}
whenever $|h|\leq|v^{-1}u|/2$.
\end{lem}
Based on the above lemma, we are able to prove the following result, which plays a key role in the proof of our main result.

\begin{lem}\label{kernel2}
Let $\rho$ be as in \eqref{rho}. Let $V\in RH_s$ with $s\geq Q/2$ and $0<\alpha<Q$. For every positive integer $N\in\mathbb N$, there exists a positive constant $C_{N,\alpha}>0$ such that for all $u,v$ and $w$ in $\mathbb H^n$, and for some fixed $0<\delta\leq 1$,
\begin{equation}\label{WH2}
\big|\mathcal K_{\alpha}(u,w)-\mathcal K_{\alpha}(v,w)\big|\leq C_{N,\alpha}\bigg[1+\frac{|w^{-1}u|}{\rho(u)}\bigg]^{-N}\frac{|v^{-1}u|^{\delta}}{|w^{-1}u|^{Q-\alpha+\delta}},
\end{equation}
whenever $|v^{-1}u|\leq |w^{-1}u|/2$.
\end{lem}
\begin{proof}
For given $\alpha\in(0,Q)$, in view of Lemma \ref{ker2} and \eqref{kauv}, we have
\begin{equation*}
\begin{split}
&\big|\mathcal K_{\alpha}(u,w)-\mathcal K_{\alpha}(v,w)\big|\\
&=\frac{1}{\Gamma(\alpha/2)}\bigg|\int_0^{\infty}P_s(u,w)\,s^{\alpha/2-1}ds-\int_0^{\infty}P_s(v,w)\,s^{\alpha/2-1}ds\bigg|\\
&\leq\frac{1}{\Gamma(\alpha/2)}\int_0^{\infty}\big|P_s(u\cdot(u^{-1}v),w)-P_s(u,w)\big|\,s^{\alpha/2-1}ds\\
&\leq\frac{1}{\Gamma(\alpha/2)}\int_0^{\infty}C_N\cdot\bigg(\frac{|u^{-1}v|}{\sqrt{s\,}}\bigg)^{\delta} s^{-Q/2}\exp\bigg(-\frac{|w^{-1}u|^2}{As}\bigg)\bigg[1+\frac{\sqrt{s\,}}{\rho(u)}+\frac{\sqrt{s\,}}{\rho(w)}\bigg]^{-N}s^{\alpha/2-1}ds\\
&\leq\frac{1}{\Gamma(\alpha/2)}\int_0^{\infty}C_N\cdot\bigg(\frac{|u^{-1}v|}{\sqrt{s\,}}\bigg)^{\delta} s^{-Q/2}\exp\bigg(-\frac{|w^{-1}u|^2}{As}\bigg)\bigg[1+\frac{\sqrt{s\,}}{\rho(u)}\bigg]^{-N}s^{\alpha/2-1}ds.
\end{split}
\end{equation*}
Arguing as in the proof of Lemma \ref{kernel}, consider two cases as below: $\sqrt{s\,}>|w^{-1}u|$ and $0\leq \sqrt{s\,}\leq|w^{-1}u|$. Then the right-hand side of the above expression can be written as $III+IV$, where
\begin{equation*}
III=\frac{1}{\Gamma(\alpha/2)}\int_{|w^{-1}u|^2}^{\infty}\frac{C_N}{s^{Q/2}}\cdot
\bigg(\frac{|u^{-1}v|}{\sqrt{s\,}}\bigg)^{\delta}\exp\bigg(-\frac{|w^{-1}u|^2}{As}\bigg)
\bigg[1+\frac{\sqrt{s\,}}{\rho(u)}\bigg]^{-N}s^{\alpha/2-1}ds,
\end{equation*}
and
\begin{equation*}
IV=\frac{1}{\Gamma(\alpha/2)}\int_0^{|w^{-1}u|^2}\frac{C_N}{s^{Q/2}}\cdot
\bigg(\frac{|u^{-1}v|}{\sqrt{s\,}}\bigg)^{\delta}\exp\bigg(-\frac{|w^{-1}u|^2}{As}\bigg)
\bigg[1+\frac{\sqrt{s\,}}{\rho(u)}\bigg]^{-N}s^{\alpha/2-1}ds.
\end{equation*}
When $\sqrt{s\,}>|w^{-1}u|$, in this case, we can deduce that
\begin{equation*}
\begin{split}
III&\leq\frac{1}{\Gamma(\alpha/2)}\int_{|w^{-1}u|^2}^{\infty}\frac{C_N}{s^{Q/2}}\cdot\bigg(\frac{|u^{-1}v|}{|w^{-1}u|}\bigg)^{\delta}
\exp\bigg(-\frac{|w^{-1}u|^2}{As}\bigg)\bigg[1+\frac{|w^{-1}u|}{\rho(u)}\bigg]^{-N}s^{\alpha/2-1}ds\\
&\leq C_{N,\alpha}\bigg[1+\frac{|w^{-1}u|}{\rho(u)}\bigg]^{-N}\bigg(\frac{|u^{-1}v|}{|w^{-1}u|}\bigg)^{\delta}
\int_{|w^{-1}u|^2}^{\infty}s^{\alpha/2-Q/2-1}ds\\
&=C_{N,\alpha}\bigg[1+\frac{|w^{-1}u|}{\rho(u)}\bigg]^{-N}\frac{|v^{-1}u|^{\delta}}{|w^{-1}u|^{Q-\alpha+\delta}},
\end{split}
\end{equation*}
where the last equality holds since $|u^{-1}v|=|v^{-1}u|$ and $0<\alpha<Q$. On the other hand, when $0\leq \sqrt{s\,}\leq|w^{-1}u|$, we can deduce that
\begin{equation*}
\begin{split}
IV&\leq C_{N,\alpha}\int_0^{|w^{-1}u|^2}\frac{1}{s^{Q/2}}\cdot\bigg(\frac{|u^{-1}v|}{\sqrt{s\,}}\bigg)^{\delta}
\bigg(\frac{|w^{-1}u|^2}{s}\bigg)^{-(Q/2+N/2+\delta/2)}\bigg[1+\frac{\sqrt{s\,}}{\rho(u)}\bigg]^{-N}s^{\alpha/2-1}ds\\
&=C_{N,\alpha}\int_0^{|w^{-1}u|^2}\frac{|u^{-1}v|^{\delta}}{|w^{-1}u|^{Q+\delta}}\bigg(\frac{\sqrt{s\,}}{|w^{-1}u|}\bigg)^{N}
\bigg[1+\frac{\sqrt{s\,}}{\rho(u)}\bigg]^{-N}s^{\alpha/2-1}ds.
\end{split}
\end{equation*}
It is easy to check that if $0\leq s\leq|w^{-1}u|^2$, then
\begin{equation*}
\frac{\sqrt{s\,}}{|w^{-1}u|}\leq\frac{\sqrt{s\,}+\rho(u)}{|w^{-1}u|+\rho(u)}.
\end{equation*}
This in turn implies that
\begin{equation*}
\begin{split}
IV&\leq C_{N,\alpha}\int_0^{|w^{-1}u|^2}\frac{|u^{-1}v|^{\delta}}{|w^{-1}u|^{Q+\delta}}\bigg[\frac{\sqrt{s\,}+\rho(u)}{|w^{-1}u|+\rho(u)}\bigg]^{N}
\bigg[\frac{\sqrt{s\,}+\rho(u)}{\rho(u)}\bigg]^{-N}s^{\alpha/2-1}ds\\
&=C_{N,\alpha}\cdot\frac{|u^{-1}v|^{\delta}}{|w^{-1}u|^{Q+\delta}}\bigg[1+\frac{|w^{-1}u|}{\rho(u)}\bigg]^{-N}\int_0^{|w^{-1}u|^2}s^{\alpha/2-1}ds\\
&=C_{N,\alpha}\bigg[1+\frac{|w^{-1}u|}{\rho(u)}\bigg]^{-N}\frac{|v^{-1}u|^{\delta}}{|w^{-1}u|^{Q-\alpha+\delta}},
\end{split}
\end{equation*}
where the last step holds because $|u^{-1}v|=|v^{-1}u|$ and $\alpha>0$. Combining the estimates of $III$ and $IV$ produces the desired inequality \eqref{WH2} for $\alpha\in(0,Q)$. This concludes the proof of the lemma.
\end{proof}

We are now in a position to give the proof of Theorem $\ref{mainthm:3}$.
\begin{proof}[Proof of Theorem $\ref{mainthm:3}$]
For given $f\in L^{p,\kappa}_{\rho,\infty}(\mathbb H^n)$ with $1\leq p<Q/{\alpha}$ and $1-{(\alpha p)}/Q\leq\kappa<1$, suppose that $f\in L^{p,\kappa}_{\rho,\theta^*}(\mathbb H^n)$, where
\begin{equation*}
\theta^*=\inf\big\{\theta>0:f\in L^{p,\kappa}_{\rho,\theta}(\mathbb H^n)\big\}\quad\mathrm{and}\quad\big\|f\big\|_{L^{p,\kappa}_{\rho,\infty}(\mathbb H^n)}=\big\|f\big\|_{L^{p,\kappa}_{\rho,\theta^*}(\mathbb H^n)}.
\end{equation*}
To prove Theorem $\ref{mainthm:3}$, it suffices to prove that the following inequality
\begin{equation}\label{end1.1}
\frac{1}{|B|^{1+\beta/Q}}\int_B\big|\mathcal I_{\alpha}f(u)-(\mathcal I_{\alpha}f)_B\big|\,du\leq C\cdot\left[1+\frac{r}{\rho(u_0)}\right]^{\vartheta}
\end{equation}
holds for any ball $B=B(u_0,r)$ with $u_0\in\mathbb H^n$ and $r\in(0,\infty)$, where $0<\alpha<Q$ and $(\mathcal I_{\alpha}f)_B$ denotes the average of $\mathcal I_{\alpha}f$ over $B$. Decompose the function $f$ as $f=f_1+f_2$, where $f_1=f\cdot\chi_{4B}$, $f_2=f\cdot\chi_{(4B)^{\complement}}$, $4B=B(u_0,4r)$ and $(4B)^{\complement}=\mathbb H^n\backslash(4B)$. By the linearity of the $\mathcal L$-fractional integral operator $\mathcal I_{\alpha}$, the left-hand side of \eqref{end1.1} can be written as
\begin{equation*}
\begin{split}
&\frac{1}{|B|^{1+\beta/Q}}\int_B\big|\mathcal I_{\alpha}f(u)-(\mathcal I_{\alpha}f)_B\big|\,du\\
&\leq\frac{1}{|B|^{1+\beta/Q}}\int_B\big|\mathcal I_{\alpha}f_1(u)-(\mathcal I_{\alpha}f_1)_B\big|\,du
+\frac{1}{|B|^{1+\beta/Q}}\int_B\big|\mathcal I_{\alpha}f_2(u)-(\mathcal I_{\alpha}f_2)_B\big|\,du\\
&:=K_1+K_2.
\end{split}
\end{equation*}
Let us consider the first term $K_1$. Applying Lemma \ref{kernel} and Fubini's theorem, we obtain
\begin{equation*}
\begin{split}
K_1&\leq\frac{2}{|B|^{1+\beta/Q}}\int_B|\mathcal I_{\alpha}f_1(u)|\,du\\
&\leq\frac{C}{|B|^{1+\beta/Q}}\int_B\bigg[\int_{4B}\bigg[1+\frac{|v^{-1}u|}{\rho(u)}\bigg]^{-N}\frac{1}{|v^{-1}u|^{Q-\alpha}}\cdot|f(v)|\,dv\bigg]\,du\\
&=\frac{C}{|B|^{1+\beta/Q}}\int_{4B}\bigg[\int_{B}\bigg[1+\frac{|v^{-1}u|}{\rho(u)}\bigg]^{-N}\frac{1}{|v^{-1}u|^{Q-\alpha}}du\bigg]\,|f(v)|\,dv\\
&\leq\frac{C}{|B|^{1+\beta/Q}}\int_{4B}\bigg[\int_{|v^{-1}u|<5r}\frac{1}{|v^{-1}u|^{Q-\alpha}}du\bigg]\,|f(v)|\,dv.
\end{split}
\end{equation*}
Hence, it follows from \eqref{radial} and H\"older's inequality that
\begin{equation*}
\begin{split}
K_1&\leq\frac{C}{|B|^{1+\beta/Q}}\int_{4B}\bigg[\int_{0}^{5r}\frac{1}{\varrho^{Q-\alpha}}\varrho^{Q-1}d\varrho\bigg]\,|f(v)|\,dv\\
&\leq\frac{C|B|^{\alpha/Q}}{|B|^{1+\beta/Q}}\bigg(\int_{4B}|f(v)|^p\,dv\bigg)^{1/p}
\bigg(\int_{4B}1\,dv\bigg)^{1/{p'}}\\
&\leq C\big\|f\big\|_{L^{p,\kappa}_{\rho,\theta^*}(\mathbb H^n)}
\cdot\frac{|B(u_0,4r)|^{{\kappa}/p+1/{p'}}}{|B(u_0,r)|^{1+\beta/Q-\alpha/Q}}\left[1+\frac{4r}{\rho(u_0)}\right]^{\theta^*}.
\end{split}
\end{equation*}
Using the inequalities \eqref{homonorm} and \eqref{2rx}, and noting the fact that $\beta/Q=\alpha/Q-{(1-\kappa)}/p$, we derive
\begin{equation*}
\begin{split}
K_1&\leq C_n\big\|f\big\|_{L^{p,\kappa}_{\rho,\theta^*}(\mathbb H^n)}\left[1+\frac{4r}{\rho(u_0)}\right]^{\theta^*}\\
&\leq C_{n,\theta}\big\|f\big\|_{L^{p,\kappa}_{\rho,\infty}(\mathbb H^n)}\left[1+\frac{r}{\rho(u_0)}\right]^{\theta^*}.
\end{split}
\end{equation*}
Let us now turn to estimate the second term $K_2$. For any $u\in B(u_0,r)$,
\begin{equation*}
\begin{split}
\big|\mathcal I_{\alpha}f_2(u)-(\mathcal I_{\alpha}f_2)_B\big|
&=\bigg|\frac{1}{|B|}\int_B\big[\mathcal I_{\alpha}f_2(u)-\mathcal I_{\alpha}f_2(v)\big]\,dv\bigg|\\
&=\bigg|\frac{1}{|B|}\int_B\bigg\{\int_{(4B)^{\complement}}\Big[\mathcal K_{\alpha}(u,w)-\mathcal K_{\alpha}(v,w)\Big]f(w)\,dw\bigg\}dv\bigg|\\
&\leq\frac{1}{|B|}\int_B\bigg\{\int_{(4B)^{\complement}}\big|\mathcal K_{\alpha}(u,w)-\mathcal K_{\alpha}(v,w)\big|\cdot|f(w)|\,dw\bigg\}dv.
\end{split}
\end{equation*}
It is clear that when $u,v\in B$ and $w\in(4B)^{\complement}$, we have
\begin{equation*}
|v^{-1}u|\leq |w^{-1}u|/2 \quad\mathrm{and}\quad |w^{-1}u|\approx |w^{-1}u_0|.
\end{equation*}
This fact along with Lemma \ref{kernel2} implies that
\begin{align*}
&\big|\mathcal I_{\alpha}f_2(u)-(\mathcal I_{\alpha}f_2)_B\big|\notag\\
&\leq\frac{C_{N,\alpha}}{|B|}\int_B\bigg\{\int_{(4B)^{\complement}}\left[1+\frac{|w^{-1}u|}{\rho(u)}\right]^{-N}
\frac{|v^{-1}u|^{\delta}}{|w^{-1}u|^{Q-\alpha+\delta}}\cdot|f(w)|\,dw\bigg\}dv\notag\\
&\leq C_{N,\alpha,n}\int_{(4B)^{\complement}}\left[1+\frac{|w^{-1}u_0|}{\rho(u)}\right]^{-N}\frac{r^{\delta}}{|w^{-1}u_0|^{Q-\alpha+\delta}}\cdot|f(w)|\,dw\notag\\
&=C_{N,\alpha,n}\sum_{k=2}^\infty\int_{2^kr\leq|w^{-1}u_0|<2^{k+1}r}
\left[1+\frac{|w^{-1}u_0|}{\rho(u)}\right]^{-N}\frac{r^{\delta}}{|w^{-1}u_0|^{Q-\alpha+\delta}}\cdot|f(w)|\,dw\notag\\
&\leq C_{N,\alpha,n}\sum_{k=2}^\infty\frac{1}{2^{k\delta}}\cdot\frac{1}{|B(u_0,2^{k+1}r)|^{1-({\alpha}/Q)}}
\int_{B(u_0,2^{k+1}r)}\left[1+\frac{2^kr}{\rho(u)}\right]^{-N}|f(w)|\,dw\notag.
\end{align*}
Furthermore, by using H\"older's inequality and \eqref{com2}, we deduce that for any $u\in B(u_0,r)$,
\begin{align}\label{end1.3}
&\big|\mathcal I_{\alpha}f_2(u)-(\mathcal I_{\alpha}f_2)_B\big|\notag\\
&\leq C\sum_{k=2}^\infty\frac{1}{2^{k\delta}}\cdot\frac{1}{|B(u_0,2^{k+1}r)|^{1-({\alpha}/Q)}}
\left[1+\frac{r}{\rho(u_0)}\right]^{N\cdot\frac{N_0}{N_0+1}}
\left[1+\frac{2^{k+1}r}{\rho(u_0)}\right]^{-N}\notag\\
&\times\bigg(\int_{B(u_0,2^{k+1}r)}\big|f(w)\big|^p\,dw\bigg)^{1/p}
\left(\int_{B(u_0,2^{k+1}r)}1\,dw\right)^{1/{p'}}\notag\\
&\leq C\big\|f\big\|_{L^{p,\kappa}_{\rho,\theta^*}(\mathbb H^n)}
\sum_{k=2}^\infty\frac{1}{2^{k\delta}}\cdot\left[1+\frac{r}{\rho(u_0)}\right]^{N\cdot\frac{N_0}{N_0+1}}
\left[1+\frac{2^{k+1}r}{\rho(u_0)}\right]^{-N}\notag\\
&\times\frac{|B(u_0,2^{k+1}r)|^{{\kappa}/p+1/{p'}}}{|B(u_0,2^{k+1}r)|^{1-\alpha/Q}}
\left[1+\frac{2^{k+1}r}{\rho(u_0)}\right]^{\theta^*}\notag\\
&=C\big\|f\big\|_{L^{p,\kappa}_{\rho,\theta^*}(\mathbb H^n)}
\sum_{k=2}^\infty\frac{|B(u_0,2^{k+1}r)|^{\beta/Q}}{2^{k\delta}}\cdot\left[1+\frac{r}{\rho(u_0)}\right]^{N\cdot\frac{N_0}{N_0+1}}
\left[1+\frac{2^{k+1}r}{\rho(u_0)}\right]^{-N+\theta^*},
\end{align}
where the last equality is due to the assumption $\beta/Q=\alpha/Q-{(1-\kappa)}/p$. From the pointwise estimate \eqref{end1.3} and \eqref{homonorm}, it readily follows that
\begin{equation*}
\begin{split}
K_2&\leq C\big\|f\big\|_{L^{p,\kappa}_{\rho,\theta^*}(\mathbb H^n)}
\sum_{k=2}^\infty\frac{1}{2^{k\delta}}\cdot\left(\frac{|B(u_0,2^{k+1}r)|}{|B(u_0,r)|}\right)^{\beta/Q}
\left[1+\frac{r}{\rho(u_0)}\right]^{N\cdot\frac{N_0}{N_0+1}}
\left[1+\frac{2^{k+1}r}{\rho(u_0)}\right]^{-N+\theta^*}\\
&\leq C\big\|f\big\|_{L^{p,\kappa}_{\rho,\infty}(\mathbb H^n)}
\sum_{k=2}^\infty\frac{1}{2^{k(\delta-\beta)}}\cdot
\left[1+\frac{r}{\rho(u_0)}\right]^{N\cdot\frac{N_0}{N_0+1}},
\end{split}
\end{equation*}
where $N\in\mathbb N$ is chosen sufficiently large such that $N\geq\theta^*$. Also observe that $\beta<\delta\leq1$, and hence the last series is convergent. Therefore,
\begin{equation*}
K_2\leq C\big\|f\big\|_{L^{p,\kappa}_{\rho,\infty}(\mathbb H^n)}\left[1+\frac{r}{\rho(u_0)}\right]^{N\cdot\frac{N_0}{N_0+1}}.
\end{equation*}
Fix this $N$ and set $\vartheta=\max\big\{\theta^*,N\cdot\frac{N_0}{N_0+1}\big\}$. Finally, combining the above estimates for $K_1$ and $K_2$, the inequality \eqref{end1.1} is proved and then the proof of Theorem \ref{mainthm:3} is finished.
\end{proof}

In the end of this article, we discuss the corresponding estimates of the fractional integral operator $I_{\alpha}=(-\Delta_{\mathbb H^n})^{-\alpha/2}$ for $0<\alpha<Q$. We denote by $K^{*}_{\alpha}(u,v)$ the kernel of the operator $I_{\alpha}$. In \eqref{claim}, we have already proved that
\begin{equation}\label{WH3}
\big|K^{*}_{\alpha}(u,v)\big|\leq C_{\alpha,n}\cdot\frac{1}{|v^{-1}u|^{Q-\alpha}}.
\end{equation}
Using the same methods and steps as we deal with \eqref{WH2} in Lemma \ref{kernel2}, we can also prove that for some fixed $0<\delta\leq 1$ and $0<\alpha<Q$, there exists a positive constant $C_{\alpha,n}>0$ such that for all $u,v$ and $w$ in $\mathbb H^n$,
\begin{equation}\label{WH4}
\big|K^{*}_{\alpha}(u,w)-K^{*}_{\alpha}(v,w)\big|\leq C_{\alpha,n}\cdot\frac{|v^{-1}u|^{\delta}}{|w^{-1}u|^{Q-\alpha+\delta}},
\end{equation}
whenever $|v^{-1}u|\leq |w^{-1}u|/2$. In view of the inequalities \eqref{WH3} and \eqref{WH4}, using the same arguments as in the proofs of Theorems \ref{mainthm:1} through \ref{mainthm:3}, we can establish the following estimates for $I_{\alpha}$, $\alpha\in(0,Q)$ on the Morrey spaces $L^{p,\kappa}(\mathbb H^n)$.

\begin{thm}\label{thm:1}
Let $0<\alpha<Q$, $1<p<Q/{\alpha}$, $0<\kappa<1-{(\alpha p)}/Q$ and $1/q=1/p-{\alpha}/{Q(1-\kappa)}$. Then the fractional integral operator $I_{\alpha}$ is bounded from $L^{p,\kappa}(\mathbb H^n)$ into $L^{q,{\kappa}}(\mathbb H^n)$.
\end{thm}

\begin{thm}\label{thm:2}
Let $0<\alpha<Q$, $p=1$, $0<\kappa<1-\alpha/Q$ and $1/q=1-{\alpha}/{Q(1-\kappa)}$. Then the fractional integral operator $I_{\alpha}$ is bounded from $L^{1,\kappa}(\mathbb H^n)$ into $WL^{q,\kappa }(\mathbb H^n)$.
\end{thm}
\begin{rem}
We point out that Theorems \ref{thm:1} and \ref{thm:2} was already obtained by Guliyev et al. in \cite[Theorem A]{guliyev}.
\end{rem}
\begin{thm}\label{thm:3}
Let $0<\alpha<Q$, $1\leq p<Q/{\alpha}$ and $1-{(\alpha p)}/Q\leq\kappa<1$. Then the fractional integral operator $I_{\alpha}$ is bounded from $L^{p,\kappa}(\mathbb H^n)$ into $\mathcal{C}^{\beta}(\mathbb H^n)$ with $\beta/Q=\alpha/Q-{(1-\kappa)}/p$ and $\beta<\delta\leq1$, where $\delta$ is given as in \eqref{WH4}.
\end{thm}
As an immediate consequence we have the following corollary.
\begin{cor}
Let $0<\alpha<Q$, $1\leq p<Q/{\alpha}$ and $\alpha p=(1-\kappa)Q$. Then the fractional integral operator $I_{\alpha}$ is bounded from $L^{p,\kappa}(\mathbb H^n)$ into $\mathrm{BMO}(\mathbb H^n)$.
\end{cor}

\end{document}